\pdfoutput=1 
\documentclass[letterpaper,11pt]{article}
\usepackage[utf8]{inputenc}
\usepackage[letterpaper,hmargin=1in,vmargin=1in,footskip=36pt]{geometry}
\DeclareMathAlphabet\mathcal{OMS}{cmsy}{m}{n}
\SetMathAlphabet\mathcal{bold}{OMS}{cmsy}{b}{n}

\usepackage[utf8]{inputenc}
\usepackage{graphicx}
\usepackage{enumerate}
\usepackage[fleqn]{amsmath}
\usepackage{amsfonts}
\usepackage{amsthm}
\usepackage{amssymb}
\usepackage{amsbsy}
\usepackage{paralist, tabularx}
\usepackage{color} 
\usepackage{arydshln}

\def\ve#1{\mathchoice{\mbox{\boldmath$\displaystyle\bf#1$}}
	{\mbox{\boldmath$\textstyle\bf#1$}}
	{\mbox{\boldmath$\scriptstyle\bf#1$}}
	{\mbox{\boldmath$\scriptscriptstyle\bf#1$}}}

\newcommand{\Z}{\ensuremath{\mathbb{Z}}}
\newcommand{\R}{\ensuremath{\mathbb{R}}}

\newcommand{\G}{\ensuremath{\mathcal{G}}}
\newcommand{\A}{\ensuremath{\mathcal{A}}}

\def\IP{({\rm IP})_{n,\veb,\vel,\veu,f}}

\def\Orthant_j{{\mathcal O}_{j}}

\newtheorem{theorem}{Theorem}
\newtheorem{claim}{Claim}
\newtheorem{corollary}{Corollary}
\newtheorem{lemma}{Lemma}

\newtheorem{observation}{Observation}
\newtheorem*{remark*}{Remark}

\newtheorem*{T2}{Theorem~\ref{thm20}}
\newtheorem*{T3}{Theorem~\ref{thm:lower-bound}}
\newtheorem*{T4}{Lemma~\ref{lemma:balance}}
\newtheorem*{T5}{Lemma~\ref{lemma:color}}
\newtheorem*{T6}{Lemma~\ref{lemma:number}}
\newtheorem*{T7}{Lemma~\ref{lemma:dec-1}}
\newtheorem*{T8}{Theorem~\ref{coro:graver}}
\newtheorem*{T9}{Theorem~\ref{thmmm}}
\newtheorem*{T10}{Theorem~\ref{them18}}
\newtheorem*{T11}{Lemma~\ref{lemma:all-balance}}

\newcommand\veb{{\ve b}}

\newcommand\vece{{\ve e}}
\newcommand\veg{{\ve g}}

\newcommand\veh{{\ve h}}
\newcommand\vel{{\ve l}}

\newcommand\vep{{\ve p}}
\newcommand\veq{{\ve q}}
\newcommand\ver{{\ve r}}

\newcommand\veu{{\ve u}}

\newcommand\vew{{\ve w}}
\newcommand\vex{{\ve x}}
\newcommand\vey{{\ve y}}
\newcommand\vez{{\ve z}}

\newcommand\veeta{{\boldsymbol{\eta}}}
\newcommand\vegamma{{\boldsymbol{\gamma}}}
\newcommand{\OO}{{\mathcal{O}}}
\newcommand{\OFPT}{{\mathcal{O}}_{FPT}}

\title{FPT Algorithms for a Special Block-structured Integer Program with Applications in Scheduling} 

\author{Hua Chen\thanks{Zhejiang University, Hangzhou, China. chenhua\_by@zju.edu.cn; zgc@zju.edu.cn}
	    \and Lin Chen \thanks{Texas Tech University, Lubbock, TX, US. chenlin198662@gmail.com} 
    \and Guochuan Zhang$^*$}

\date{\today}

\begin{document}
\maketitle

\begin{abstract}
We consider integer programs (IPs) whose constraint matrix has a special block structure. More precisely, we consider IP:   $\min\{f(\vex): {H}_{\textnormal{com}} \vex=\veb, \vel\le \vex\le \veu, \vex\in \Z^{
t_B+nt_A} \}$, in which the objective function $f$ is separable convex and the constraint matrix ${H}_{\textnormal{com}}$ is composed of small submatrices $A_i,B,C,D_i$ such that the first row of ${H}_{\textnormal{com}}$ is $(C,D_1,D_2,\ldots,D_n)$, the first column of ${H}_{\textnormal{com}}$ is $(C,B,B,\ldots,B)^{\top}$, the main diagonal of ${H}_{\textnormal{com}}$ is $(C,A_1,A_2,\ldots,A_n)$, and the rest entries are 0. Furthermore, the rank of submatrix $B$ is 1.

%the submatrix $B$ has only one row, or it contains multiple rows but only one row is nonzero,$B=(\ver_1,\ve 0,\ldots,\ve 0)^{\top}$, where $\ver_1\neq \ve 0$. 

We study fixed parameter tractable (FPT) algorithms by taking as parameters the number of rows and columns of small submatrices, together with the largest absolute value over their entries. 

We call the IP studied (almost) combinatorial 4-block $n$-fold IP. It generalizes the generalized $n$-fold IP and is meanwhile a special case of the generalized 4-block $n$-fold IP. In the literature, existing FPT algorithms for block-structured IP rely on bounding the $\ell_1$- or $\ell_\infty$-norm of elements of the Graver basis. The existence of FPT algorithms for 4-block $n$-fold IP is a major open problem and Chen et al. [ESA 2020] showed that the $\ell_\infty$-norm of the Graver basis elements of 4-block n-fold IP is $\Omega(n)$. This motivates us to study special cases of the generalized 4-block $n$-fold IP to find structural insights. 

We show that, the $\ell_{\infty}$-norm of the Graver basis elements of combinatorial 4-block $n$-fold IP is also $\Omega(n)$. However, there exists some FPT-value $\lambda$ such that for any nonzero element $\veg\in \{\vex:H_{\textnormal{com}}\vex=\ve 0\}$, $\lambda\veg$ can always be decomposed into Graver basis elements in the same orthant whose $\ell_{\infty}$-norm is FPT-bounded (while $\veg$ itself might not admit such a decomposition). This seems to exhibit an ``intermediate'' phenomenon. Based on this, we are able to bound the $\ell_{\infty}$-norm of Graver basis elements for combinatorial 4-block $n$-fold IP by $\OFPT(n)$ and develop an  $\OFPT(n^4\hat{L}^2)$-time algorithm (here the $\OFPT$ hides a multiplicative FPT-term, and $\hat{L}$ denotes the logarithm of the largest number occurring in the input).

As applications, we show that combinatorial 4-block $n$-fold IP can be used to model important generalizations of the classical scheduling problems, including scheduling with rejection and bicriteria scheduling, which implies that our FPT algorithm establishes a general framework to settle the classical scheduling problems.
\\


\end{abstract}

\clearpage

\setcounter{page}{1}

\section{Introduction}
\label{sec:typesetting-summary}
Integer programs (IPs) whose constraint matrix has a special block structure have received a considerable attention in recent years. As an important subclass of the general IP, it finds applications in a variety of optimization problems including scheduling~\cite{chen2018covering,jansen2018empowering,knop2020combinatorial}, routing~\cite{chen2018covering}, stochastic integer multi-commodity flows~\cite{hemmecke2010polynomial}, stochastic programming with second-order dominance constraints~\cite{gollmer2011note}, etc.

First, we consider a block-structured IP as follows:
\begin{equation}\label{ILP:2}
	\IP: \quad \min\{f(\vex): {H}_{\textnormal{com}} \vex=\veb, \, \vel\le \vex\le \veu,\, \vex\in \Z^{t_B + nt_A} \}, 
\end{equation}
where $f: \mathbb{R}^{t_B+nt_A}\rightarrow \mathbb{R}$ is a separable convex function, and ${H}_{\textnormal{com}}$ consists of small submatrices $A_i$, $B$, $C$ and $D_i$ as follows:
\begin{eqnarray}\label{eq:matrix}
{H_{\textnormal{com}}}:=
\begin{pmatrix}
C & D_1 & D_2 & \cdots & D_n \\
B & A_1 & 0  &   & 0  \\
B & 0  & A_2 &   & 0  \\
\vdots &   &   & \ddots &   \\
B & 0  & 0  &   & A_n
\end{pmatrix}, \hspace{15mm}
{H}:=
\begin{pmatrix}
C & D_1 & D_2 & \cdots & D_n \\
B_1 & A_1 & 0  &   & 0  \\
B_2 & 0  & A_2 &   & 0  \\
\vdots &   &   & \ddots &   \\
B_n & 0  & 0  &   & A_n
\end{pmatrix} 
 \enspace .
\end{eqnarray}
Here, $A_i$'s (or $B$ or $C$ or $D_i$'s, resp.) are $s_A\times t_A$ (or $s_B\times t_B$ or $s_C\times t_C$ or $s_D\times t_D$, resp.) matrices, and furthermore, the rank of matrix $B$ is 1. 

Note that when $C=B=0$, the above problem reduces to the generalized $n$-fold IP. In the meantime, IP~\eqref{ILP:2} is a special case of the generalized 4-block $n$-fold IP~\cite{hemmecke2014graver} where the constraint matrix $H$ consists of submatrices $A_i$, $B_i$, $C$ and $D_i$ as Eq~\eqref{eq:matrix}. It is worth mentioning that the overall structure of $H$ implies that $s_C=s_D$, $s_A=s_B$, $t_B=t_C$ and $t_A=t_D$.

Let $\Delta$ be the largest absolute value among all the entries of $A_i,B,C$ and $D_i$. The goal of this paper is to study FPT algorithms for combinatorial 4-block $n$-fold IP by taking $\Delta$, $s_A,s_B,s_C,s_D$ and $t_A,t_B,t_C,t_D$ as parameters, i.e., we aim for an algorithm that runs polynomially in $n$. 

When $s_B=s_A=1$, we call IP~\eqref{ILP:2} {\it combinatorial 4-block $n$-fold IP} (and $H_{\textnormal{com}}$ {\it combinatorial 4-block $n$-fold matrix}) as it generalizes the combinatorial $n$-fold IP studied in~\cite{knop2020combinatorial} (combinatorial $n$-fold IP can be viewed as a special case where $C=B=0$ and all the entries of $A_i$'s are 1).  

To be consistently, when the rank of matrix $B$ is 1, IP~\eqref{ILP:2} is called \emph{almost combinatorial $4$-block $n$-fold IP}. To tackle this problem, first we are focused on combinatorial $4$-block $n$-fold IP while $s_B=s_A=1$. Then we show that all results achieved remain true for almost combinatorial $4$-block $n$-fold IP. 

There are two facts that make (almost) combinatorial 4-block $n$-fold IP an interesting subclass of the general block-structured IP.

From an application point of view, combinatorial 4-block $n$-fold IP generalizes combinatorial $n$-fold IP and thus offers a stronger tool for optimization problems. In particular, %Knop et al.~\cite{knop2018scheduling} has demonstrated the close relationship between the classical scheduling problems and $n$-fold IP. With combinatorial 4-block $n$-fold IP, we are able to model a broader class of scheduling problems, namely scheduling with rejection and bi-objective scheduling problems, and derive FPT algorithms.  
%As an application, we consider parallel machine scheduling problems. 
Knop and Kouteck{\`y}~\cite{knop2018scheduling} modeled parallel machine scheduling problems $R||C_{\max}$ and $R||\sum_{\ell}w_\ell C_\ell$ as $n$-fold IPs and developed FPT algorithms (parameters include the largest job processing time, different types of machines and different types of jobs). Utilizing combinatorial $4$-block $n$-fold IP, we are able to model a broader class of scheduling problems and derive similar FPT algorithms. Specifically, we consider two generalizations of the classical scheduling model. One is the bicriteria scheduling problem $R||\theta C_{\max}+\sum_{\ell}w_\ell C_\ell$, which considers the combination of two common scheduling objectives. The other is the scheduling problem with job rejection $R||C_{\max}+E$, where jobs can be rejected at a certain cost and the goal is to minimize the scheduling cost plus the total rejection cost. The reader may refer to Section~\ref{appli} for the precise definitions of the two problems and the corresponding FPT algorithms.

From a theoretical point of view, combinatorial 4-block $n$-fold IP exhibits an interesting \lq\lq intermediate\rq\rq\, phenomenon in its Graver basis (see Section~\ref{sec:pre} for the definition). As we will provide more details later in the related work, FPT algorithms have been developed for several special cases of the generalized 4-block $n$-fold IP (see, e.g.,~\cite{aschenbrenner2007finiteness,cslovjecsek2021block,hemmecke2013n,jansen2018empowering,koutecky2018parameterized}). All of these algorithms rely on the fact that the $\ell_{\infty}$-norm (or even $\ell_{1}$-norm) of the Graver basis elements for these special cases are bounded by some FPT-value. Unfortunately, Chen et al.~\cite{chen2020new} showed very recently that the $\ell_{\infty}$-norm of Graver basis elements for 4-block $n$-fold IP is $\Omega(n)$. It thus becomes a challenging problem that without the boundedness of $\ell_{\infty}$-norm, what other properties can we expect from the Graver basis elements which may lead to an FPT algorithm? In this paper, we observe an interesting phenomenon:
On the one hand, the $\ell_{\infty}$-norm of the Graver basis elements for combinatorial 4-block $n$-fold IP is still $\Omega(n)$ even if $s_D=1$ (see Theorem~\ref{thm:lower-bound}). On the other hand, Graver basis elements whose $\ell_{\infty}$-norm is bounded by some FPT-value seem to be strong enough for the purpose of decomposition. More precisely, we have the following Theorem~\ref{11-n}, which states that for some fixed $\lambda$ and any  $\veg \in\ker_{\Z}(H_{\textnormal{com}})$, $\lambda \veg$ can always be decomposed into the summation of Graver basis elements with $\ell_{\infty}$-norm bounded by some FPT-value. Interestingly, this $\lambda$ only depends on $t_B$ and $\Delta$. 

\begin{theorem}\label{11-n}
 Let $H_{\textnormal{com}}$ be a combinatorial $4$-block $n$-fold matrix. Then there exists a positive integer $\lambda\le 2^{2^{2^{\OO(t_B^2\log (t_B\Delta))}}}$ (which is only dependent on $t_B$ and $\Delta$) such that for any $\ve g\in\ker_{\Z}(H_{\textnormal{com}})$, we have $\lambda\veg=\veg_1+\veg_2+\cdots+\veg_p$ for some $p\in \mathbb{Z}_{>0}$ and $\veg_j\in \ker_{\Z}(H_{\textnormal{com}})$, and furthermore, $\veg_j\sqsubseteq \lambda\veg$ and $\|\veg_j\|_\infty=2^{2^{\OO(t_A\log\Delta+s_Dt_D\log\Delta)}\cdot 2^{2^{\OO(t_B^2\log\Delta)}}}$.
\end{theorem}
Here the upper bounds for $\lambda$ and $\|\veg_j\|_{\infty}$'s are triply exponential in the parameters.

Utilizing Theorem~\ref{11-n}, we are able to show that $\|\veg\|_{\infty}= \OFPT(n)$ for any Graver basis element $\veg$, and develop an algorithm of running time $\OFPT(n^4\hat{L}^2)$ for combinatorial 4-block $n$-fold IP, where 
$\OFPT$ hides a multiplicative factor that only depends on $\Delta,s_A,s_B,s_C,s_D,t_A,t_B,t_C,t_D$, and $\hat{L}$ denotes the logarithm of the largest number occurring in the input. The special feature implied by Theorem~\ref{11-n} as well as our techniques may be of separate interest for a broader class of IPs.

\begin{remark*} Theorem~\ref{11-n} and our FPT algorithm for combinatorial $4$-block $n$-fold IP remain true for almost combinatorial $4$-block $n$-fold IP. Such a generalization allows submatrices $A_i$'s to contain multiple rows subject to that these rows are \lq\lq local constraints\rq\rq. It is, however, not clear whether Theorem~\ref{11-n} still holds if we allow the $n$ submatrices $B\in\Z^{1\times t_B}$ to be different.

\end{remark*}

\subparagraph*{Related work.}
The existence of FPT algorithms for the generalized 4-block $n$-fold IP (where the constraint matrix is given by $H$ in Eq~\eqref{eq:matrix}) remains as one major open problem in the area of integer programming. However, important progress has been achieved in recent years on its special cases. 
In particular, extensive research has been carried out on three fundamental subclasses -- 4-block $n$-fold IP, the generalized $n$-fold IP and the generalized two-stage stochastic IP. 

When $A_i=A$, $B_i=B$ and $D_i=D$, the generalized 4-block $n$-fold reduces to 4-block $n$-fold IP, which has been studied before mainly by Hemmecke et al.~\cite{hemmecke2014graver} and Chen et al.~\cite{chen2020new}. In particular, Chen et al.~\cite{chen2020new} showed that the infinity norm of Graver basis elements for such 4-block $n$-fold IP is bounded by $\min\{n^{\OO(t_A^2)},n^{\OO(s_D)}\}$, and developed a $\min\{n^{\OO(t_A^2t_B)},n^{\OO(s_Dt_B)}\}$-time algorithm. Consequently, their results do not yield FPT algorithms for combinatorial 4-block $n$-fold IP. Very recently Chen et al.~\cite{chen2020blockstructured} studied 4-block $n$-fold IP when $\Delta$ is {\em not} part of the parameters, and proved that when $t_A=s_A+1$ and $\textnormal{rank}(A)=s_A$, 4-block $n$-fold IP can be solved in $(t_A+t_B)^{O(t_A+t_B)}\cdot n^{O(t^2_A)}\cdot poly(\log\Delta)$ time.

%Two well-known special cases of the (generalized) 4-block $n$-fold IP include the (generalized) $n$-fold IP where $C=B_i=0$ for all $i$, and the (generalized) two-stage stochastic IP where $C=D_i=0$, as shown in the following matrices.
\iffalse

\begin{eqnarray}
{H}^{\textnormal{n-fold}}:=
\begin{pmatrix}
D_1 & D_2 & \cdots & D_n \\
A_1 & 0  &   & 0  \\
0  & A_2 &   & 0  \\
\vdots &    & \ddots &   \\
0  & 0  &   & A_n
\end{pmatrix}\hspace{15mm}
{H}^{\textnormal{two-stage}}:=
\begin{pmatrix}
B_1 & A_1 & 0  &   & 0  \\
B_2 & 0  & A_2 &   & 0  \\
\vdots &   &   & \ddots &   \\
B_n & 0  & 0  &   & A_n
\end{pmatrix}. \label{matrix:1}
\end{eqnarray}
\fi

%For the 4-block $n$-fold IP, Hemmecke et al.~\cite{hemmecke2014graver} proved that $\ell_\infty$-norm of Graver elements is $\OFPT(n^{2^{s_D}})$ , which gives rise to an algorithm of running time $\OFPT(n^{t_B\cdot 2^{s_D}})$. 
%Very recently, Chen et al.~\cite{chen2020new} provided an improved algorithm of running time $\OFPT(n^{s_Dt_B})$. However, it remains an important open problem whether there exist FPT algorithms for the 4-block $n$-fold IP.

When $C=B_i=0$ for all $i$, the generalized 4-block $n$-fold IP reduces to the generalized $n$-fold IP, and we denote the constraint matrix as ${H}^{\textnormal{n-fold}}$. This IP was initialized by De Loera et al.~\cite{de2008n}. In 2013, Hemmecke et al.~\cite{hemmecke2013n} developed the first FPT algorithm. Later on, a series of researches have been carried out to further improve its running time%and also to extend the algorithm to the generalized $n$-fold IP
~\cite{altmanova2019evaluating,cslovjecsek2021block,eisenbrand2018fastera,eisenbrand2019algorithmic,jansen2018empowering,jansen2019near}.
%on dynamic programming with time of  $n^3t_A^3\cdot (s_Ds_A\Delta)^{\mathcal{O}(t_A^2s_D)}\cdot encode(IP)$ for generalized $n$-fold IP, which had been applied to computational social choice~\cite{knop2020voting} and scheduling~\cite{knop2018scheduling}. 
%Furthermore, Jansen et al.~\cite{jansen2018empowering} made use of $n$-fold IP to gain efficient algorithms for scheduling problems. 
Most recently, Cslovjecsek et al.~\cite{cslovjecsek2021block} presented an algorithm of running time $2^{\OO(s^2_As_D)}(s_Ds_A\Delta)^{\OO(s_A^2+s_As_D^2)} (nt_A)^{1+o(1)}$ for the generalized $n$-fold IP.

When $C=D_i=0$ for all $i$, the generalized 4-block $n$-fold IP reduces to the generalized two-stage stochastic IP, and we denote the constraint matrix as ${H}^{\textnormal{two-stage}}$. This IP was first studied by Hemmecke and Schultz~\cite{hemmecke2003decomposition} and Aschenbrenner and Hemmecke~\cite{aschenbrenner2007finiteness}. Their result was improved by in a series of subsequent papers~\cite{eisenbrand2019algorithmic,jansen2021double,klein2021complexity,koutecky2018parameterized}. 
The current best-known algorithm for the generalized two-stage stochastic IP runs doubly exponential in the parameters $\Delta, s_A, t_B$ by Klein~\cite{klein2021complexity}.

\section{Notations and Preliminaries}\label{sec:pre}
%\subparagraph*{Notation.}
%In $4$-block $n$-fold IP, $\vex=(\vex^0,\vex^1,\vex^2,\ldots,\vex^n)$ where $\vex^0$ corresponds to $B$. And then in $n$-fold IP $\vex^0$ vanishes and the vector becomes $(\vex^1,\vex^2,\ldots,\vex^n)$.
\subparagraph*{Notations.} We write column vectors in boldface, e.g., $\vex, \vey$, and their entries in normal font, e.g., $x_i, y_i$. If $\vex\in \Z^{d_1}$ and $\vey\in \Z^{d_2}$, then we abuse the notation by using $(\vex,\vey)$ to denote a column vector in $\Z^{d_1+d_2}$.
Recall that a solution $\vex$ for $4$-block $n$-fold IP is a $(t_B+nt_A)$-dimensional column vector, and we write it into $n+1$ \emph{bricks}, such that $\vex=(\vex^0,\vex^1,\cdots,\vex^n)$ where $\vex^0 \in \Z^{t_B}$ and each $\vex^i \in \Z^{t_A}$, $1\le i\le n$. We call $\vex^i$ the \emph{$i$-th brick} for $0\le i\le n$. For a vector or a matrix, we write $\|\cdot\|_{\infty}$ to denote the maximal absolute value of its elements. For two vectors $\vex,\vey$ of the same dimension, $\vex\cdot\vey$ denotes their inner product.  We use $[i]$ to represent the set of integers $\{1,2,\cdots,i\}$, and $[i:j]$ for $\{i,i+1,\cdots,j\}$ where $i<j$.

Two vectors $\vex$ and $\vey$ are called {\it sign-compatible} if $x_i\cdot y_i\ge 0$ holds for every pair of coordinates $(x_i,y_i)$. Recall the matrix $H_{\textnormal{com}}$ in Eq~\eqref{eq:matrix}. We denote by $H_{\textnormal{com}}^{\textnormal{n-fold}}$ the submatrix obtained from $H_{\textnormal{com}}$ by removing the first column $(C,B,B,\cdots,B)^{\top}$, and $H_{\textnormal{com}}^{\textnormal{two-stage}}$ the submatrix obtained by removing the first row $(C,D_1,\cdots,D_n)$.

%We usually use lowercase letters for variables and uppercase letters for matrices. 

%\smallskip
%\noindent\textbf{Input size.} In an IP~\eqref{ILP}, it is allowed that the entries of $\veb, \vel,\veu$ are $\infty$. However, utilizing the techniques of Tardos~\cite{tardos1986strongly}, Koutecký et al.~\cite{koutecky2018parameterized} showed that without loss of generality we can restrict that $\|\veb\|_\infty, \|\vel\|_\infty, \|\veu\|_\infty\le 2^{O(n\log n)}\Delta^{O(n)}$. We assume this bound throughout this paper.

Throughout this paper, we use $\OFPT(1)$ to represent a parameter that depends only on $\Delta,s_A,s_B,s_C,s_D,t_A,t_B,t_C,t_D$ where $\Delta$ is the maximal absolute value among all the entries of $A_i,B_i,C,D_i$. In other words, $\OFPT(1)$ is only dependent on the small matrices $A_i,B_i,C,D_i$ and is independent of $n$. For any computable function $g(x)$, we write $\OFPT(g)$ to represent a computable function $g'(x)$ such that $|g'(x)|\le \OFPT(1) \cdot |g(x)|$.

%\subparagraph*{Fixed parameter tractable  (FPT) algorithm.} For a problem instance $I$ with a \emph{parameter} $k$, we call an algorithm with running time $f(k)poly(|I|)$ an \emph{FPT} algorithm, where $f$ is a computable function. See also the textbook~\cite{cygan2015parameterized}. 

\subparagraph*{Graver basis.}
%Consider the integer linear programming~\eqref{ILP}. 
We define $\sqsubseteq$ to be the \emph{conformal order} in $\mathbb{R}^d$ such that $\vex\sqsubseteq\vey$ if  %$\vex$ and $\vey$ lie in the same orthant, i.e., $x_i\cdot y_i\ge 0$ for each$i=1,...,d$, 
$\vex$ and $\vey$ are sign-compatible and $|x_i|\le |y_i|$ for each $i=1,...,d$.  Given any subset $X\subseteq \mathbb{R}^d$, we say $\vex$ is a $\sqsubseteq$-\emph{minimal} element
of $X$ if $\vex \in X$ and there does not exist $\vey \in X, \vey\neq \vex$ such that $\vey\sqsubseteq\vex$. It is known that every subset of $\Z^d$ has
finitely many $\sqsubseteq$-minimal elements.

 Then the \emph{Graver basis} (\cite{graver1975foundations}) of an integer matrix $\A$ is defined the finite set $\G(\A)$, which consists of all $\sqsubseteq$-minimal elements of $\text{ker}_{\Z}(\A)\backslash \{\ve0\}$, where $\text{ker}_{\Z}(\A)=\{\vex\in \Z^{\tilde{N}}|\A\vex=\ve0\}$.

 \subparagraph*{Graver-best augmentation.} Consider a general IP \begin{eqnarray}\label{ILP}
 \min\{f(\vex): \A \vex=\veb, \vel\le \vex\le \veu, \vex\in \Z^d \},
 \end{eqnarray}
 We call $\vex$ a feasible solution if $\A\vex=\veb$ and $\vel \le \vex\le \veu$.  Given a feasible solution $\vex$ to IP~\eqref{ILP}, we call $\veg$ a \emph{feasible step} if $\vex +\veg$ is feasible for the IP. Furthermore, if $f(\ve x +\veg ) < f ( \vex )$,  a feasible step $\veg$ is called \emph{augmenting}.   An augmenting step $\veg$ and a step length $\rho \in \Z $ form an $\vex$-feasible step pair with respect to a feasible solution
 $\vex$ if $\vel \le \vex+\rho\veg\le \veu$. An augmenting step $\veh$ with $\rho_0\in \Z$ is a \emph{Graver-best step} for $\vex$ if $f(\vex+\rho_0\veh) \le f(\vex+\rho\veg)$ for all
 $\vex$-feasible step pairs $(\veg,\rho) \in \G(\A)\times\Z$. 
 
 The \emph{Graver-best augmentation procedure} for an IP and a given feasible solution $\vex_0$ work as follows:
 
 1. If there is no Graver-best step for $\vex_0$, return it as optimal.
 
 2. If a Graver-best step $(\veh,\rho)$ for $\vex_0$ exists, set $\vex_0:=\vex_0+\rho\veh$ and go to 1.
 
 The following Lemma~\ref{lemma21} tells us that it is sufficient to focus all our attention on finding Graver-best steps. 
 
 %\begin{definition}[Graver-best oracle]\label{defi1}
 %A \emph{Graver-best oracle} for an integer matrix $\A$ is one that, %queried on
 %$\vew,\veb,\vel,\veu$ and $\vex$ feasible to \emph{IP~\eqref{ILP}}, %returns a Graver-best step $\veh$ for $\vex$.
 %\end{definition}
 
 %\begin{lemma}[\cite{koutecky2018parameterized}]
 %\label{lemma21}
 %	Given a Graver-best oracle for $\A$, \emph{IP~\eqref{ILP}} can be solved in strongly polynomial oracle time.
 %\end{lemma}
 \begin{lemma}[\cite{de2012algebraic}, implicit in Theorem 3.4.1]
 	\label{lemma21}
 	Given a feasible solution $\vex_0$, and a separable convex function $f$, the Graver-best augmentation procedure finds the optimum in at most $(2n-2)\log (\hat{f})$ steps, where $\hat{f}=f(\ve x_0)-f(\ve x^*)$ for some integer optimum $\ve x^*$.
 \end{lemma}
 
 \begin{lemma}[\cite{onn2010nonlinear}, Lemma 3.2]
 	\label{lemma22}
 	Every integer vector $\vex\neq0$ with $\A\vex = \ve0$ is a sign-compatible sum $\vex=\sum_{i}\veg_i$ of Graver basis elements $\veg_i \in \G(\A)$, with some elements possibly appearing with repetitions. 
 \end{lemma}
\begin{theorem}[\cite{klein2021complexity}, Theorem 2]
	\label{lemma23}
	Let $\veg$ be a Graver element of a generalized two-stage stocastic IP with constraint matrix ${H}^{\textnormal{two-stage}}$. %as Eq~\eqref{matrix:1}. 
	Then $\|\veg\|_\infty\le g_\infty({H}^{\textnormal{two-stage}})$, where $g_\infty({H}^{\textnormal{two-stage}})$ only depends on $s_B,t_B,\Delta$ and $g_\infty({H}^{\textnormal{two-stage}})\le 2^{2^{\OO(s_Bt_B^2\log (s_B\Delta))}}$. 
\end{theorem}

By Lemma~\ref{lemma22} and Theorem~\ref{lemma23}, we conclude that for the constraint matrix of a generalized two-stage stocastic IP, any one basis $\vex$ satisfying ${H}^{\textnormal{two-stage}}\vex= 0$ can be decomposed into a sign-compatible sum of Graver bases; i.e., $\vex=\sum_{i}\veg_i$ where $\veg_i\in \G({H}^{\textnormal{two-stage}})$, and $\|\veg_i\|_\infty\le2^{2^{\OO(s_Bt_B^2\log (s_B\Delta))}}$. 
 \subparagraph*{The Steinitz Lemma~(\cite{grinberg1980value,steinitz1913bedingt}).}
 Let an arbitrary norm be given in $\R^d$, and let $\vex_1,\ldots,\vex_m\in \mathbb{R}^d$ with $\|\vex_i\|\le \zeta$ for $i=1,\ldots,m$. If $\sum_{i=1}^{m}\vex_i=\vex$. Then there is a permutation $\pi$ such that for each $\ell\in \{1,\ldots,m\}$ the norm of the partial sum $\|\sum_{i=1}^{\ell}\vex_{\pi(i)}-\frac{\ell-d}{m}\vex\|\le d\zeta$.

%By Lemma~\ref{lemma22} and Theorem~\ref{lemma23}, we conclude that for the constraint matrix of a generalized two-stage stocastic IP, any one basis $\vex$ satisfying ${H}^{\textnormal{two-stage}}\vex=\ve0$ can be decomposed into a sign-compatible sum of Graver bases; i.e., $\vex=\sum_{i}\veg_i$ where $\veg_i \in \G({H}^{\textnormal{two-stage}})$, and $\|\veg_i\|_\infty\le (s_Bt_B\Delta)^{O(s_Bt_B(s_B\Delta)^{s_Bt_B^2})}$.
 %and will be utilized to bound the $\ell_{\infty}$-norm of Graver basis elements.

 The above Steinitz Lemma is commonly used to bound the $\ell_{\infty}$-norm of Graver basis elements, and is also used in our paper. In particular, Lemma~\ref{lemma:merging-lemma}  follows from the Steinitz Lemma.

% Particularly, if $\sum_{i=1}^{m}\vex_i=\ve 0$, then there is a permutation $\pi$ such that for each $\ell\in \{1,\ldots,m\}$ the norm of the partial sum $\|\sum_{i=1}^{\ell}\vex_{\pi(i)}\|$ is bounded by $d\zeta$.

%In what follows, we broadly rely on the Steinitz Lemma and some  known properties above about Graver basis to show our desired results. 
\begin{lemma}[\cite{chen2020new}] \label{lemma:merging-lemma}
	Let $\vex_1,\vex_2,\cdots,\vex_m$ be a sequence of vectors in $\mathbb{Z}^d$ such that $\vex=\sum_{i=1}^m\vex_i$, and $\|\vex_i\|_{\infty}\le \zeta$. Then the set $[m]$ can be partitioned into $m'$ subsets $T_1,T_2,\cdots,T_{m'}$ satisfying that: $\cup_{j=1}^{m'}T_j=[m]$, and for every $1\le j\le m'$ it holds that $\sum_{i\in T_j} \vex_i\sqsubseteq \vex$, $|T_j|\le (c\zeta)^{d^2}$ for some constant $c$. In particular, if $d=1$, then $|T_j|\le 6\zeta+2$ for all $j$.
\end{lemma}

 When applying combinatorial 4-block $n$-fold IP to solve optimization problems, we may establish IPs with the constraint being $H_{\textnormal{com}}\vex\le \veb$. The following observation ensures that such a constraint can be transformed to a standard form Eq~\eqref{ILP:2} without destroying the structure of the constraint matrix.
\begin{observation}\label{obs}
	Considering the \emph{IP}  
	$\min\{ f(\ve x): H_{\textnormal{com}} \vex\le \veb, \vel\le \vex\le \veu, \vex\in \Z^{t_B+nt_A} \},$ 
	where $H_{\textnormal{com}}$ is defined as in~\eqref{eq:matrix} 
	and $B\in\Z^{1\times t_B}$, we can make the constraint $H_{\textnormal{com}} \vex\le \veb$ tight ($H^{\prime}_{\textnormal{com}} \vex = \veb$) by adding $n(s_D+1)$ slack variables in total and keep the new constraint matrix $H^{\prime}_{\textnormal{com}}$ being in the form of~\eqref{eq:matrix}. Specifically, we write the constraints in~\eqref{eq:matrix} as follows:
	\begin{eqnarray}
	&& C\vex^0+  \sum_{i=1}^{n}D_i\vex^i\le \ve b^0\label{cons1}\\
	&& B\vex^0+A_i\vex^i\le \ve b^i, \hspace{15mm} \forall 1\le i\le n \label{cons2}
	\end{eqnarray}
	For Constraint~\eqref{cons1}, notice that there are in fact $s_D$ inequalities, and for each inequality, we add $n$ slack variables. Similarly, since $B$ is a vector, Constraint~\eqref{cons2} includes $n$  inequalities, and for each inequality, we add 1 slack variable. Thus we have the new constraint matrix $H^{\prime}_{\textnormal{com}}$ in~\eqref{eq2:matrix}.
	\begin{eqnarray}\label{eq2:matrix}
	{H_{\textnormal{com}}}=
	\begin{pmatrix}
	C & D_1 & D_2 & \cdots & D_n \\
	B & A_1 & 0  &   & 0  \\
	B & 0  & A_2 &   & 0  \\
	\vdots &   &   & \ddots &   \\
	B & 0  & 0  &   & A_n
	\end{pmatrix}, \hspace{15mm} 
	{H^{\prime}_{\textnormal{com}}}=
	\begin{pmatrix}
	C & D^{\prime}_1 & D^{\prime}_2 & \cdots & D^{\prime}_n \\
	B & A^{\prime}_1 & 0  &   & 0  \\
	B & 0  & A^{\prime}_2 &   & 0  \\
	\vdots &   &   & \ddots &   \\
	B & 0  & 0  &   & A^{\prime}_n
	\end{pmatrix} \enspace,
	\end{eqnarray}
	where $D^{\prime}_i$ has dimension $s_D\times (t_A+1+s_D)$ and $A^{\prime}_i$ has dimension $1\times (t_A+1+s_D)$, and 
	\[D^{\prime}_i=
	\begin{array}{c@{\hspace{-5pt}}l}
	\left(\begin{array}{c;{2pt/2pt}ccccc}
	& 0 & 1 & 0 & \cdots & 0 \\
	& 0 & 0 & 1 &   & 0\\
	\raisebox{2ex}[0pt]{\large \text{$D_i$}}
	& \vdots & \vdots &  & \ddots & \\
	& 0 & 0 & 0 &   & 1 \\
	\end{array}\right)
	&
	\\[-5pt]
	\begin{array}{cc}
	\textcolor{white}{\large D}& \underbrace{\rule{24mm}{0mm}}_{1+s_D}
	\end{array}
	\end{array},\hspace{15mm} 
	A^{\prime}_i=(A_i,\underbrace{1,0,\ldots,0}_{1+s_D}).
	\]
	%For the following applications of the combinatorial 4-block $n$-fold IP to the scheduling problem in this paper, we won't write out the tight form $H^{\prime}_{\textnormal{com}}$ any more. 
\end{observation}

\section{Structural Results for Combinatorial 4-block $n$-fold}\label{sec:structure}
The goal of this section is to prove Theorem~\ref{11-n}, based on which in Section~\ref{sec:alg} we will be able to bound the $\ell_{\infty}$-norm of the Graver basis elements of combinatorial 4-block $n$-fold IP, and design an FPT algorithm using the iterative augmentation framework developed in a series of prior research works (see Graver-best augmentation in Section~\ref{sec:pre}).

Towards the proof of Theorem~\ref{11-n}, we first give an example.
\subparagraph*{Example.} Let $H_0$ be a 4-block $n$-fold matrix where $C=(-1,-1,-1)$, $D=(5,3)$, $B=(0,-1,1)$ and $A=(3,4)$. Let $\veg=(\veg^0,\veg^1,\cdots,\veg^n)$ such that $\veg^0=(1,n-1,n)$ and $\veg^i=(1,-1)$ for all $i$ (see the left side of Eq~\eqref{eq:11decompose} where $\veg$ is written explicitly). It is not difficult to verify that $H_0\veg =\ve 0$. Moreover, we are able to prove that $\veg$ is a Graver basis element, thus proving Theorem~\ref{thm:lower-bound} (see Appendix~\ref{appsec:lower-bound} for the omitted proof).

%Chen et al.~\cite{chen2020new} showed that in general, the $\ell_{\infty}$-norm of the Graver basis of the 4-block $n$-fold IP cannot be bounded by $\OFPT(1)$. In particular, they showed that the $\ell_{\infty}$-norm of the Graver basis of a 4-block $n$-fold IP can be $n^{\Omega(t)}$ when $s_D,s_B=\Theta(t)$. Nevertheless, this does not rule out the possibility that the $\ell_{\infty}$-norm of the Graver basis may be bounded by $\OFPT(1)$ when $s_D$ and $s_B$ are small, in particular, when $s_B=1$. Unfortunately, we have the following.  

%\subparagraph*{An instance.}
\begin{theorem}\label{thm:lower-bound}
 There exists a 4-block $n$-fold IP where $s_B=s_D=1$ such that $\|\veg\|_{\infty}=\Omega(n)$ for some Graver basis element $\veg$.
\end{theorem}

Despite that the $\veg$ constructed in the proof cannot be decomposed into ``thin" kernel elements in the same orthant, we observe that, interestingly, by multiplying $\veg$ with some small value (bounded by $\OFPT(1)$), such a decomposition follows. More precisely, we have the following.

%On the other hand, we observe that while $\veg$ cannot be decomposed into ``thin'' kernel elements in the same orthant, by multiplying $\veg$ with some small value (which is 11), such a decomposition follows:

%We take $\lambda=11$ and observe the following:
\begin{eqnarray}\label{eq:11decompose}
11\times \begin{pmatrix}
	1   \\
	n-1   \\
	n\\
	1   \\
	-1   \\
	1  \\
	-1   \\
	\vdots\\
	1   \\
	-1   \\
	1   \\
	-1  
\end{pmatrix}
=
\begin{pmatrix}
	0   \\
	0   \\
	11\\
	3   \\
	-5   \\
	3   \\
	-5   \\
	\vdots\\
	3   \\
	-5   \\
	7   \\
	-8   
\end{pmatrix}
+\begin{pmatrix}
	0   \\
	11   \\
	11\\
	8   \\
	-6   \\
	0   \\
	0   \\
	\vdots\\
	0   \\
	0   \\
	0   \\
	0   
\end{pmatrix}
+\begin{pmatrix}
	0   \\
	11   \\
	11\\
	0   \\
	0   \\
	8   \\
	-6   \\
	0   \\
	0   \\
	\vdots\\
	0   \\
	0   
\end{pmatrix}
+\cdots+
\begin{pmatrix}
	0   \\
	11   \\
	11\\
	0   \\
	0   \\
	\vdots\\
	0   \\
	0   \\
	8   \\
	-6   \\
	0   \\
	0   
\end{pmatrix}
+
\begin{pmatrix}
	11   \\
	0   \\
	0\\
	0   \\
	0   \\
	0   \\
	0   \\
	\vdots\\
	0   \\
	0   \\
	4   \\
	-3  
\end{pmatrix}.
\label{eq50}
\end{eqnarray}
Notice that there are in total $n+1$ vectors on the right side of Eq~\eqref{eq50} and let them be $\veg_1,\veg_2,\cdots,\veg_{n+1}$: Among $\veg_j$'s the first vector $\veg_1=(0,0,11,3,-5,\cdots,3,-5,7,-8)$ consists of $\veg_1^0=(0,0,11)$, $n-1$ copies of $\veg_1^i=(3,-5)$ and one copy of $\veg_1^n=(7,-8)$. Each of $\veg_2$ to $\veg_n$ consists of $(0,11,11)$, one copy of $(8,-6)$ and 0's. And the last vector $\veg_{n+1}$ consists of $(11,0,0)$, one copy of $(4,-3)$ and 0's. It is easy to verify that $\veg_j\sqsubseteq 11\veg$ and $H_0\veg_j=\ve 0$. 
\subparagraph*{A high level overview on the proof of Theorem~\ref{11-n}.}  Recall that $H_{\textnormal{com}}$ is a combination of two submatrices, the first row $(C,D_1,\cdots,D_n)$ and a two-stage stochastic matrix $H_{\textnormal{com}}^{\textnormal{two-stage}}$. Therefore, any $\veg\in \ker_{\Z}(H_{\textnormal{com}})$ also satisfies that $\veg\in \ker_{\Z}(H_{\textnormal{com}}^{\textnormal{two-stage}})$, and by Theorem~\ref{lemma23} for any $\lambda\in\Z_{>0}$ we have $\lambda \veg = \sum_{j=1}^L\ve\xi_j$  where $\|\ve\xi_j\|_{\infty}=\OFPT(1)$, $\ve\xi_j\sqsubseteq \lambda\veg$ and $\ve\xi_j\in \ker_{\Z}(H_{\textnormal{com}}^{\textnormal{two-stage}})$. Note that $(C,D_1,\cdots,D_n)\ve\xi_j$ is not necessarily $\ve0$ and hence $\ve\xi_j$'s may not belong to $\ker_{\Z}(H_{\textnormal{com}})$. To show $\lambda\veg$ can be decomposed into sign-compatible elements of $\ker_{\Z}(H_{\textnormal{com}})$ with bounded $\ell_{\infty}$-norm, it suffices to show that if $L$ (and consequently $\|\lambda\veg\|_{\infty}$) is too huge, then there exists some $\veeta\in \ker_{\Z}(H_{\textnormal{com}})$ such that $\veeta\sqsubseteq \lambda \veg$ and $\|\veeta\|_{\infty}=\OFPT(1)$. Afterwards, we proceed to decompose $\lambda\veg-\veeta$. 
A natural idea to construct such an $\veeta$ is to select a subset $S$ with an $\OFPT(1)$ number of $\ve\xi_j$'s such that $(C,D_1,D_2,\cdots,D_n)\sum_{j\in S}\ve\xi_j=\ve 0$. Unfortunately,  the cardinality of $S$ needs to be $\Omega(n)$ to make $(C,D_1,D_2,\cdots,D_n)\sum_{j\in S}\ve\xi_j=\ve 0$ observed by Chen et al.~\cite{chen2020new}. To bypass this obstacle, $\ve\eta$ needs to be constructed in a way more \lq\lq flexible\rq\rq\,  than a direct summation of $\ve\xi_j$'s. Thus, 
we try to enable a \lq\lq cross-position\rq\rq\, construction, that is, we will allow each brick $\veeta^i$ to consist of bricks from different positions of $\ve\xi_j$'s, e.g., $\veeta^i=\ve\xi_{j_1}^{i_1}+\ve\xi_{j_2}^{i_2}$ where $i_1,i_2$ may be different from $i$. This will cause a critical problem. Suppose $\veeta^i=\ve\xi_{j_1}^{i_1}+\ve\xi_{j_2}^{i_2}$ and $\veeta^{i'}=\ve\xi_{j_3}^{i_3}+\ve\xi_{j_4}^{i_4}$, then 
how should we set the value of $\veeta^0$ to ensure that $B\veeta^0+A_i\veeta^i=\ve 0$? We observe that, if the decomposition $\lambda \veg = \sum_{j=1}^L\ve\xi_j$ satisfies that $B\ve\xi_{j}^0$ equals the same value for all $j$ (called \emph{the uniform condition}), and additionally if it holds that $A_i=A_{i_1}=A_{i_2}$ and $A_{i'}=A_{i_3}=A_{i_4}$, then by setting $\veeta^0=\ve\xi_{j_1}^0+\ve\xi_{j_2}^0$ (or equivalently, $\veeta^0=\ve\xi_{j_3}^0+\ve\xi_{j_4}^0$) we have $B\ve\eta^0+A_i\veeta^i=B\ve\xi^0_{j_1}+A_{i_1}\ve\xi_{j_1}^{i_1}+B\ve\xi^0_{j_2}+A_{i_2}\ve\xi_{j_2}^{i_2}=\ve 0$, and similarly $B\ve\eta^0+A_{i'}\veeta^{i'}=0$. That means, \lq\lq cross-position\rq\rq\, construction is possible if the uniform condition is met. Unfortunately, the uniform condition is not necessarily true. Only for combinatorial 4-block $n$-fold IP and some suitably chosen $\lambda=\OFPT(1)$ we can guarantee the uniform condition (nevertheless, our proof remains true for almost combinatorial 4-block $n$-fold IP, as we discuss at the end of Section~\ref{subsec:thm-proof}.) 

With the uniform condition, the construction of $\veeta$ still has two major challenges. One is that $\veeta$ must satisfy $(C,D_1,\cdots,D_n)\veeta=\ve 0$. We will generalize the Steinitz Lemma to a \lq\lq colorful\rq\rq\, variant to handle it (see Lemma~\ref{lemma:color}). The other challenge is more fundamental and is due to \lq\lq cross-position\rq\rq\, construction itself. Say, e.g., $\veeta^i=\ve\xi_{j_1}^{i_1}+\ve\xi_{j_2}^{i_2}$. While we know $\ve\xi_{j_1}^{i_1}\sqsubseteq \lambda\veg^{i_1}$ given that $\ve\xi_{j_1}\sqsubseteq \lambda\veg$, it is not necessary that $\ve\xi_{j_1}^{i_1}\sqsubseteq \lambda\veg^{i}$. How can we select the right bricks so that $\veeta^i\sqsubseteq \lambda\veg^i$ for all $i$? Indeed, is it even possible or not? Towards this, our rough idea is as follows: we consider every coordinate of $\lambda\veg^i$. If one coordinate is sufficiently large (larger than some threshold $\sigma=\OFPT(1)$), then the summation of any $\OFPT(1)$ bricks $\ve\xi_j$'s should never exceed it. Otherwise, $\veeta^i\sqsubseteq \lambda\veg^i$ may be violated and this coordinate becomes critical. We will introduce a hierarchy over $\lambda\veg^i$'s depending on each of its coordinate being critical or not, and the \lq\lq cross-position\rq\rq\, construction will only be carried out for positions (e.g., $i_1$ and $i_2$ in $\veeta^i=\ve\xi_{j_1}^{i_1}+\ve\xi_{j_2}^{i_2}$) in the same level under the hierarchy. We will show that, by doing so, if $\|\lambda\veg\|_{\infty}$ is sufficiently large, then $\veeta\sqsubseteq \lambda\veg$ can be guaranteed through a counting argument.

The remainder of this section is devoted to the proof of Theorem~\ref{11-n}. Towards this, we first introduce some concepts.

Consider the generalized 4-block $n$-fold IP with constraint matrix $H$ and let ${\veg}\in \ker_{\Z}(H)$ be an arbitrary kernel element. A decomposition $\veg=\sum_{j=1}^N\veeta_j$ is called \emph{uniform}, if for all $j$ it holds that $\veeta_j\sqsubseteq \veg$, $H^{\textnormal{two-stage}}\veeta_j=\ve 0$, and moreover, there is some fixed $\veq\in\Z^{s_B}$, $\veq\neq \ve0$, such that for any $j\in [N]$,
\begin{eqnarray}
	&&B_i\veeta_j^0=\ve 0, \forall i\in [n] \qquad\textnormal{or} \qquad B_i\veeta_j^0=\veq, \forall i\in [n] .\label{eq5-extra}
\end{eqnarray}
That is, for all $i$ and $j$, $B_i\veeta_j^0$ may only take two possible values. For each $j$, $B_i\veeta_j^0$ must be the same for all $i$.  We say $\veeta_j$ is {\it tier-0} if $B_i\veeta_j^0=\ve 0$, and $\veeta_j$ is {\it tier-1} if $B_i\veeta_j^0=\veq$. Consequently, $A_i\veeta_j^i=\ve 0$ for all $i$ or $A_i\veeta_j^i=-\veq$ for all $i$. 

In case of combinatorial 4-block $n$-fold IP, $B_i=B$ and $s_B=1$, and hence Eq~\eqref{eq5-extra} is simplified such that $B\veeta_j^0$ is either $0$ or $q$ for all $j$. 

Consider an arbitrary $\bar{\veg}\in \ker_{\Z}(H)$ that admits a uniform decomposition $\bar{\veg}=\sum_{j=1}^{\bar{N}}\bar{\veeta}_j$ such that $\|\bar{\veeta}_j\|_{\infty}\le \bar{\eta}_{\max}$. As each $\bar{\veeta}_j$ is either tier-0 or tier-1, we denote by ${\bar{N}}_0$ (or ${\bar{N}}_1$) the number of tier-0 (or tier-1) vectors among $\bar{\veeta}_1$ to $\bar{\veeta}_{\bar{N}}$. We say that the decomposition is $\omega$-balanced if ${\bar{N}}_0\le \omega {\bar{N}}_1$, and exact $\omega$-balanced if the equality holds. In particular, we define that $\ve0$ admits an $\omega$-balanced uniform decomposition. 

\begin{lemma}\label{lemma:balance}
For any $\bar{\veg}\in \ker_{\Z}(H)$, if $\bar{\veg}$ admits a uniform decomposition $\bar{\veg}=\sum_{j=1}^{\bar{N}}\bar{\veeta}_j$ where $\|\bar{\veeta}_j\|_{\infty}\le \bar{\eta}_{\max}$, then there exists ${\veg}\in \ker_{\Z}(H)$ such that ${\veg}\sqsubseteq \bar{\veg}$, $B_i(\bar{\veg}^0-\veg^0)=\ve 0$ for all $i\in [n]$, and ${\veg}$ admits an $\omega$-balanced uniform decomposition for $\omega\le (\Delta t_D \bar{\eta}_{\max})^{\OO(s_D^2)}$. Moreover, if $\bar{\veg}-\veg\neq 0$, then we have $\bar{\veg}-\veg=\veg_1+\veg_2+\cdots+\veg_p$ for some $p\in \mathbb{Z}$ and $\veg_j\in \ker_{\Z}(H)$, and furthermore, $\veg_j\sqsubseteq \bar{\veg}-\veg$ and $\|\veg_j\|_\infty\le (\Delta t_D \bar{\eta}_{\max})^{\OO(s_D^2)}$.
\end{lemma}

\begin{remark*}
If $\bar{\veg}^0=\ve0$, then Lemma~\ref{lemma:balance} holds for $\veg=\ve0$. 
\end{remark*}

It suffices to focus on a balanced uniform decomposition. Further notice that if $\bar{\veeta}_{j_1}$ is tier-1 and $\bar{\veeta}_{j_2}$ is tier-0, then $\bar{\veeta}_{j_1}+\bar{\veeta}_{j_2}$ is tier-1. Hence, we have the following.

\begin{lemma}\label{lemma:all-balance}
If ${\veg}=\sum_{j=1}^{\bar{N}}\bar{\veeta}_j$ is an $\omega$-balanced uniform decomposition where $\bar{\eta}_{\max}=\max_{j\in[\bar{N}]}\|\bar{\veeta}_j\|_{\infty}$, then ${\veg}$ admits a uniform decomposition ${\veg}=\sum_{j=1}^{N}{\veeta}_j$ such that every $\veeta_j$ is tier-1, and ${\eta}_{\max}=\max_{j\in[{N}]}\|{\veeta}_j\|_{\infty}\le \omega\bar{\eta}_{\max}$.
\end{lemma}

%if $\bar{\veg}=\sum_{j=1}^{\bar{N}}\bar{\veeta}_j$ is an $\omega$-balanced uniform decomposition, by merging some of the $\bar{\veeta}_j$'s, we can obtain an uniform decomposition $\bar{\veg}=\sum_{j=1}^{\bar{N}'}\bar{\veeta}_j'$ with $N_0'$ tier-0 summand and $N_1'$ tier-1 summand, and furthermore, $N_0'=\omega' N_1'$ for some $\omega'\le \omega$ (here $\omega'$ can be $0$). We call it exact $\omega'$-balanced decomposition.

We will prove the following Lemma~\ref{lemma:sub} in Section~\ref{subsec:lemma}. Then we show the existence of such decomposition for combinatorial 4-block $n$-fold IP, thus concluding Theorem~\ref{11-n} in Section~\ref{subsec:thm-proof}.

\begin{lemma}\label{lemma:sub}
Suppose ${\veg}\in \ker_{\Z}(H)$ admits a uniform decomposition ${\veg}=\sum_{j=1}^{N}{\veeta}_j$ such that $\|\veeta_j\|_{\infty}\le \eta_{\max}$, and every $\veeta_j$ is tier-1. There exists $\tau=(\Delta\eta_{\max})^{\Delta^{\OO(s_At_A+s_Dt_D)}}$ such that {if  $\|\veg\|_{\infty}> \tau$}, then there exists $\veeta\in \ker_{\Z}(H)$ such that $\veeta\sqsubseteq {\veg}$ and $\|\veeta\|_\infty\le \tau$, and furthermore, $\veeta^{0}=\sum_{j\in S} \veeta_j^0$ for some $S\subseteq [N]$.
\end{lemma}

\subsection{Proof of Lemma~\ref{lemma:sub}}\label{subsec:lemma}
%Towards the proof of Lemma~\ref{lemma:sub}, we first introduce some useful notions. 
%\subsection{Preprocessing}
%Given 

\subsubsection{A hierarchical structure over bricks of \text{${\veg}$}}\label{subsec:notion}
%In the following we show $\lambda\veg$ can be decomposed into $\sqsubseteq$-conformable kernel elements with bounded $\ell_{\infty}$-norm. 
%Note that each $\ve\eta_j=(\ve\eta_j^0,\ve\eta_j^1,\cdots,\ve\eta_j^n)$ consists of $n+1$ bricks. To prove Lemma~\ref{lemma:sub}, our basic idea is to show that if $N$ is sufficiently large (larger than some $\OFPT(1)$-value), then among the $N(n+1)$ bricks we can always select a few of them to construct a new vector $\ve\eta$ such that (i). $H_{\textnormal{com}}^{\textnormal{two-stage}}\ve\eta=0$; (ii). $(C,D_1,\cdots,D_n)\ve\eta=0$; and (iii). $\ve\eta\sqsubseteq \lambda \veg$; thus proving Theorem~\ref{11-n}. To satisfy the three properties, we cannot pick an arbitrary $\veeta_j^i$ and use it as the $i'$-th brick of $\ve\eta$. We will present lemmas that can guide us in the selection of bricks in Subsection~\ref{subsec:lemma}. Towards that, we need to introduce some notions.   
As we describe in the overview, we will construct $\veeta$ in Lemma~\ref{lemma:sub} from the bricks $\veeta_j^i$'s via \lq\lq cross-position\rq\rq\, construction. For each $\veeta^i$, there will be some restrictions regarding which brick $\veeta_j^{i'}$ can be used, indicated by the hierarchical structure we introduce in the following.

We first observe that as $A_i$'s and $D_i$'s are small submatrices with the largest coefficient bounded by $\Delta$, there are in total at most $\Delta^{\OO(s_At_A+s_Dt_D)}$ different kinds of $A_i$'s and $D_i$'s, and hence $\varphi\le \Delta^{\OO(s_At_A+s_Dt_D)}=\OFPT(1)$ different pairs of $(A_i,D_i)$. By re-indexing, we may divide $[n]$ into $\varphi$  subsets as $[n]=\bigcup_{k=1}^{\varphi} [r_{k-1}+1:r_k]$ where $0=r_0<r_1<r_2<\cdots<r_{\varphi}=n$ such that $(A_i,D_i)$'s are identical for every $i\in [r_{k-1}+1:r_k]$.
\iffalse
Consequently, $H_{\textnormal{com}}$ can be written as follows:
\begin{eqnarray}
H_{\textnormal{com}}=\left(	\begin{array}{ccccccccccc}
	C     & D_{r_1} &  \cdots    & D_{r_1}    &    D_{r_2}  & \cdots  & D_{r_2} & \cdots  & D_{r_\varphi}  & \cdots & D_{r_\varphi} \\
		B      & A_{r_1} &      &     &      &   &  &   &   &  &  \\
		\vdots &     &\ddots&     &      &   &  &   &   &  &  \\
		B      &     &      & A_{r_1} &      &   &  &   &   &   &   \\	
		B      &     &      &     &  A_{r_2} &   &  &   &   &   &  \\
		\vdots &     &      &     &      &  \ddots &  &   &   &  &  \\
			B  &     &      &     &      &   & A_{r_2} &   &   &  &  \\
		\vdots &     &      &     &      &   &   & \ddots  &   &  &  \\
			B  &     &      &     &      &   &   &   & A_{r_\varphi}  &  &  \\
		\vdots &     &      &     &      &   &   &   &   & \ddots &   \\
		B      &     &      &     &      &   &  &   &   &  &A_{r_\varphi}
\end{array} \right). \hspace{0.5cm} 
	\end{eqnarray}
\fi
%We first introduce some terms. Recall that in $H_{\textnormal{com}}$ we have $(A_i,D_i)=(A_{r_k},D_{r_k})$ for all $i\in [r_{k-1}+1:r_k]$. 
Let $I_0=\{0\}$ and $I_k=[r_{k-1}+1:r_k]$.  For simplicity we let $D_{r_0}=D_0=C$, then $(C,D_1,D_2,\cdots,D_n)\veeta_j=\sum_{k=0}^{\varphi}\sum_{i\in I_k}D_{r_k}\veeta_j^i$.
%\begin{eqnarray}\label{eq:1.5}
 % (C,D_1,D_2,\cdots,D_n)\veeta_j=\sum_{k=0}^{\varphi}\sum_{i\in I_k}D_{r_k}\veeta_j^i.  
%\end{eqnarray}

%When it is clear from context, we may also abuse the notation by calling $D_{r_k}\veeta_j^i$, $i\in I_k$, as a zone-$\nu$ brick. 
We define {\it type} and {\it subtype} for integer vectors.  {Let $\sigma$ be some sufficiently large value (it suffices to take $\sigma= (\Delta \eta_{\max})^{\Delta^{\OO(s_At_A+s_Dt_D)}}$ as we will explain later}). %where $\varphi=\Delta^{t_A+s_Dt_A}$ and $\lambda_0=(s_Bt_B\Delta)^{O(s_Bt_B(s_B\Delta)^{s_Bt_B^2})}=(t_B\Delta)^{O(t_B\Delta^{t_B^2})}$. 
We classify each integer $x$ into one of the five {\it types}: 
\begin{itemize}
    \item $0$, if $x=0$,
    \item close-positive, if $1\le x\le \sigma$,
    \item faraway-positive, if $x>\sigma$,
    \item close-negative, if $-\sigma\le x\le -1$, and
    \item faraway-negative, if $x<-\sigma$.
\end{itemize}

%We extend the definition to vectors. $d$-dimensional vectors can be classified into $5^d$ types such that two vectors $\vex$ and $\vey$ belong to the same type as a vector if and only if for every $1\le i\le d$, the $i$-th coordinate of $\vex$ and $\vey$ have the same type as an integer.

We can further classify all integers into $2\sigma+3$ {\it subtypes} by sub-dividing the type close-positive (or close-negative) into $\sigma$ categories, that is, $x$ is called of subtype-$x$ if $-\sigma\le x\le \sigma$. %and is called faraway-positive or faraway-negative if $x>\sigma$ or $x<-\sigma$, respectively. 

We now extend the definitions of types and subtypes to vectors. All $d$-dimensional vectors can be classified into $5^d$ types (or $(2\sigma+3)^d$ subtypes) such that two vectors $\vex$ and $\vey$ belong to the same type (or subtype) as a vector if and only if for every $1\le \ell\le d$, the $\ell$-th coordinate of $\vex$ and $\vey$ have the same type (or subtype) as an integer.

%Consider the set of all bricks $\{\ve\eta_j^i:0\le i\le n, 1\le j\le N\}$. We introduce a hierarchy on the subsets of bricks.

Now we classify the indices $0\le i\le n$ based on ${\veg}$ as follows:
\begin{itemize}
    \item Megazone. Each $I_k$, $0\le k\le \varphi$ is called a megazone.
    %Each brick $\ve\eta_j^i$ where $i\in I_k$ is called a {\it megazone}-$k$ brick. 
    There are $\varphi+1=\Delta^{\OO(s_At_A+s_Dt_D)}$ megazones.
    \item Zone. A megazone is sub-divided into zones such that indices $i,i'$ belong to the same zone if and only if they belong to the same megazone and ${\veg}^i$ and ${\veg}^{i'}$ have the same type. There are at most $1+5^{t_A}\varphi=\Delta^{\OO(s_At_A+s_Dt_D)}$ different zones. For $0\le \nu\le 5^{t_A}\varphi$, let $\beta_\nu\in\Z_{\ge 0}$ be the number of indices belonging to zone-$\nu$.
    \item Subzone. A zone is sub-divided into subzones so that indices $i,i'$ belong to the same subzone if and only if they belong to the same zone and ${\veg}^i$ and ${\veg}^{i'}$ have the same subtype. There are at most $(2\sigma+3)^{t_A}\cdot (1+5^{t_A}\varphi)$ subzones. For $0\le \iota\le (2\sigma+3)^{t_A} (1+5^{t_A}\varphi)-1$, let $\gamma_{\iota}\in\Z_{\ge 0}$ be the number of indices belonging to subzone-$\iota$.
    \item Slot. Every index $0\le i\le n$ is called a slot. There are $n+1$ slots.
\end{itemize}

Figure~\ref{fig1} in Appendix~\ref{figure} illustrates the relationships among megazones, zones and subzones. 
It is remarkable that the number of zones, $1+5^{t_A}\varphi$, is independent of the value of $\sigma$. $\sigma$ only comes into play at subzone level, which is crucial to our proof. Further, note that megazone-0 only contains one zone, and this zone contains one subzone, and this subzone contains one slot, which is slot-0. For simplicity, we let slot-0 be in subzone-0 and zone-0.

%We call all the bricks $\ve\eta_j^i$ where $i\in I_k$ as {\it megazone}-$k$ bricks (jobs) for $0\le k\le \varphi$. Each megazone-$k$ is further divided into $\OFPT(1)$ zones. 

%Note that each brick $\ve\eta_j^i$ is either $t_B$-dimensional (when $i=0$) or $t_A$-dimensional (when $i\in [n]$), we call each coordinate of a brick as a {\it bit}. 

For ease of description, we will take a viewpoint of the {\it Scheduling} problem. We view each brick $\veeta^i_j$ as a job and there are $N(n+1)$ jobs. We assume there are $n+1$ machines (from machine~0 to machine~$n$), and think of each job $\veeta_j^i$ as a job originally scheduled on machine~$i$. Machines can be divided into megazones, zones and subzones based on their indices. A job (brick) that is originally scheduled on a machine in megazone-$k$ (or zone-$\nu$ or subzone-$\iota$, resp.) is called a megazone-$k$ (or zone-$\nu$ or subzone-$\iota$, resp.) job (brick). %Furthermore, jobs are divided into two {\it tiers}. Recall that job $\veeta_j^i$ is tier-0 if $B_i\veeta^0_j=\ve0$ for all $i$, and is tier-1 if $B_i\veeta^0_j=\veq$ for all $i$. 
We add up jobs on each machine just like adding up vectors, whereas the load of machine~$i$ in the original schedule is ${\veg}^i$. 

Constructing a new vector $\veeta\sqsubseteq {\veg}$ is like rescheduling jobs. That is, we remove jobs from machines in the original schedule, and then select and re-assign a subset of suitable jobs to machines. By doing so, we obtain a partial schedule. The load of machine~$i$ in the partial schedule, which is the summation of jobs assigned to it, will be $\veeta^i$. We will take ${\veg}^i$ as the {\it capacity} of machine~$i$. If the summation of several jobs equals $\vex\sqsubseteq{\veg}^i$, then we say the jobs fit machine~$i$. 

To prove Lemma~\ref{lemma:sub}, we need to construct a partial schedule $\ve\eta$ such that (i) $H^{\textnormal{two-stage}}\ve\eta=\ve 0$, (ii) $(C,D_1,\cdots,D_n)\ve\eta=\ve 0$ and (iii) $\ve\eta\sqsubseteq  \veg$. %as we step by step achieve it in the following.
In the following Subsection~\ref{subsec:prop-1}, Subsection~\ref{subsec:prop-2}, and Subsection~\ref{subsec:prop-3}, we will identify the conditions for the partial schedule to satisfy each property respectively, and finalize the proof of Lemma~\ref{lemma:sub} in Subsection~\ref{subsec:final-lemma}. 

%In the following the terminologies bricks and jobs, as well as the vector $\veeta$ (where each $\veeta^i$ is the summation of several bricks $\veeta_j^{i'}$) and a partial schedule is used interchangeably. 

%Bricks and jobs will be used interchangeably. We call all the $(n+1)N$ bricks from $\ve\eta_j$'s as jobs. 

%\smallskip
%\noindent\textbf{\large Step 1: Selecting jobs to satisfy property (i) - $H^{\textnormal{two-stage}}\ve\eta=0$.}%\label{subsec:prop-1}
\subsubsection{Selecting jobs to satisfy property (i) - $H^{\textnormal{two-stage}}\ve\eta=\ve 0$}\label{subsec:prop-1}

Recall that $A_i$'s are the same for $i$ in each megazone (and hence in each zone). For $\nu\ge 1$, let machine~$i$ be an arbitrary zone-$\nu$ machine and $\veeta_{j_1}^{i'}$ be an arbitrary zone-$\nu$ job. Then %$A_i\veeta_{j_1}^{i'}=0$ if job $\veeta_{j_1}^{i'}$ is tier-0, or 
$A_i\veeta_{j_1}^{i'}=-\veq$ by the definition in Eq~\eqref{eq5-extra}. If we put one zone-$\nu$ job $\veeta_{j_1}^{i'}$ on machine~$i$ and meanwhile put one zone-$0$ job $\veeta_{j_2}^0$ on machine~$0$, then it holds that $B_i\veeta_{j_2}^0+A_i\veeta_{j_1}^{i'}=\ve 0$. Hence, we have the following observation.

%According to Eq~\eqref{eq5-extra} and the fact that $A_i=A_{r_k}$ for $i\in [r_{k-1}+1,r_k]$, we have the following straightforward observation.
\begin{observation}\label{obs:prop-i}
Let $h$ be an arbitrary non-negative integer. Let $\ve\eta=(\veeta^0,\veeta^1,\cdots,\veeta^n)$ be a partial schedule where we assign arbitrary $h$ jobs in zone-$\nu$ to each zone-$\nu$ machine (i.e., for every $i$ in zone-$\nu$, $\veeta^i$ is the summation of $h$ zone-$\nu$ jobs). Then $H^{\textnormal{two-stage}}\ve\eta=\ve 0$.
\end{observation}

\subsubsection{Selecting jobs to satisfy property (ii) - $(C,D_1,\cdots,D_n)\ve\eta=\ve 0$}\label{subsec:prop-2}
%\noindent\textbf{\large Step 2: Selecting jobs to satisfy property (ii) - $(C,D_1,\cdots,D_n)\ve\eta=0$}\label{subsec:prop-2}

Recall that $D_0=C$ and $(D_0,D_1,\cdots,D_n)\sum_{j=1}^N\veeta_j=\ve 0$, which is a long sequence of addition consisting of $(n+1)N$ summands. We are interested in a subsequence whose sum is $0$ and meanwhile respects Observation~\ref{obs:prop-i}, that is, we want to select exactly $h\beta_{\nu}$ jobs from zone-$\nu$ such that their sum (after multiplying corresponding $D_{i}$'s) is $0$ (while recall that there are exactly $\beta_{\nu}$ zone-$\nu$ machines). Towards this, we first prove the following lemma, which gives a \lq\lq colorful\rq\rq\, version of the Steinitz Lemma.

\begin{lemma}\label{lemma:color}
Let $\vex_1,\ldots,\vex_{M}\in \mathbb{Z}^d$ be a sequence of vectors such that $\|\vex_i\|_{\infty}\le \zeta$ for some $\zeta\ge 1$ and every $i=1,\ldots,M$. Furthermore, there are $\mu$ colors, and each vector $\vex_i$ is associated with one color. There are in total $\alpha_j\overline{m}$ vectors of color $j$ where $\alpha_j,\overline{m}\in\Z_{>0}$ and $\sum_{j=1}^{\mu}\alpha_j=\alpha$, $M=\alpha \overline{m}$. Supposing that $\sum_{i=1}^{M}\vex_i=\ve0$ and $M$ is sufficiently large (i.e., $M>(2d\zeta+2\mu\zeta+1)^{d+\mu}\alpha+\alpha+d+\mu$), then among $\vex_1,\cdots,\vex_{M}$ we can find $\alpha_jm$ vectors of each color $j$ such that their summation is $\ve0$, and $m\le (2d\zeta+2\mu\zeta+1)^{d+\mu}$. 
\end{lemma}
By the Steinitz Lemma, it is easy to see the existence of a subset of vectors that add up to $\ve0$. Lemma~\ref{lemma:color} further indicates that the number of vectors of each color in this subset is proportional to their number in the whole set of $M$ vectors. Notice that $\overline{m}$ and $m$ are independent with each other. $\overline{m}$ may be very large, while $m$ can be bounded by an FPT-value. 
See the proof in Appendix~\ref{appsec:color}. 

{
Now we apply Lemma~\ref{lemma:color} to the equation $(D_0,D_1,\cdots,D_n)\sum_{j=1}^N\veeta_j=\sum_{i,j,\ell} D_\ell\veeta_j^i=\ve 0$ as follows. If $i$ belongs to some zone-$\nu$ (which further belongs to some megazone-$k$), %and $j$ belongs to tier-0 (or tier-1), 
then we take each summand $D_{\ell}\veeta_j^i$ as a vector in $\Z^{s_D}$ of color $\nu$. Consequently, we have in total $1+5^{t_A}\varphi=\Delta^{\OO(s_At_A+s_Dt_D)}$ different colors, and $M=(n+1)N$ vectors where the number of vectors in each color $\nu$ is $N\beta_\nu$. Further notice that $\|D_{\ell}\veeta_j^i\|_{\infty}\le t_D\Delta\eta_{\max}$. Hence, as long as $M=(n+1)N>\rho (n+1)$ for $\rho=(\Delta \eta_{\max})^{\Delta^{\OO(s_At_A+s_Dt_D)}}$, we can always find out $m\beta_\nu$ summands in color $\nu$ (corresponding to $m\beta_\nu$ jobs in zone-$\nu$) such that $m\le(\Delta \eta_{\max})^{\Delta^{\OO(s_At_A+s_Dt_D)}}$, and they sum up to $\ve0$. %Furthermore, $\max_{i,j,k}\|D_{r_k}\veeta_j^i\|_{\infty}$, $s_D$ and the number of colors $1+5^{t_A}\varphi$ are independent of $\sigma$, thus 
%Further notice that $m$ and $\rho$ are also independent of $\sigma$.
Moreover, Lemma~\ref{lemma:color} can be applied iteratively until there are fewer than $\rho(n+1)$ jobs left. }
Our argument above implies the following.

\begin{lemma}\label{obs:prop-ii}
There exist some $m,\rho=(\Delta \eta_{\max})^{\Delta^{\OO(s_At_A+s_Dt_D)}}$ such that if $N>\rho$, then all the $(n+1)N$ jobs (bricks) can be divided into $\lfloor \frac{N-\rho}{m}\rfloor +1 :=\psi+1$ groups such that 
\begin{compactitem}
    \item Except the last group, each group consists of $\beta_\nu m$ zone-$\nu$ jobs for all $\nu$.
    \item The last group consists of $\beta_\nu m'$ zone-$\nu$ jobs where ${m}'\le \rho+m$.
    \item If we evenly distribute jobs in every group to machines such that a zone-$\nu$ machine is assigned $m$ jobs (or $m'$ jobs if it is the last group), then the partial schedule $\veeta$ satisfies that $(C,D_1,\cdots,D_n)\veeta=\ve 0$. 
\end{compactitem}
%Furthermore, $m$ and $\rho$, and hence $\psi$, are independent of $\sigma$.
\end{lemma}

\begin{remark*} Note that the number of zones, and thus $m,\rho,\psi$, are all independent of $\sigma$. We pick $\sigma\ge (\rho+m)\eta_{\max}$ which guarantees that when we evenly distribute jobs in each group to machines, the infinity norm of their sum never exceeds $\sigma$.
\end{remark*}

%since $\max_{i,j,k}\|D_{r_k}\veeta_j^i\|_{\infty}$, $s_D$ and the number of colors $1+5^{t_A}\varphi$ are independent of $\sigma$, we know $m$ and $\rho$, and hence $\psi$, are also independent of $\sigma$.

Notice that since we assign the same number of zone-$\nu$ jobs to zone-$\nu$ machines, by Observation~\ref{obs:prop-i} the partial schedule $\veeta$ in Lemma~\ref{obs:prop-ii} also satisfies that $H^{\textnormal{two-stage}}\veeta=\ve0$, and hence $H\veeta=\ve0$.

\subsubsection{Selecting jobs to satisfy property (iii) - $\ve\eta\sqsubseteq {\veg}$}\label{subsec:prop-3}
%\smallskip
%\noindent\textbf{\large Step 3: Selecting jobs to satisfy property (iii) - $\ve\eta\sqsubseteq {\veg}$.}\label{subsec:prop-3}

According to Lemma~\ref{obs:prop-ii}, by evenly distributing jobs to machines in each zone, every group of jobs induces a partial schedule $\veeta$. We show in this subsection that if there are sufficiently many groups, %(i.e., when $\psi=\lfloor \frac{N-\rho}{m}\rfloor$), 
then there must be a group which induces $\veeta\sqsubseteq {\veg}$. For simplicity we ignore the last group and focus on remaining groups.

We first briefly argue why evenly distributing jobs to machines in each zone in an arbitrary way may generate a partial schedule that is $\not\sqsubseteq{\veg}$. Note that when we apply Lemma~\ref{lemma:color} to divide jobs into groups, we can only guarantee there are $\beta_\nu m$ jobs from each zone-$\nu$ (and hence every machine in zone-$\nu$ can get exactly $m$ jobs in zone-$\nu$), but we cannot guarantee there are $\gamma_{\iota} m$ jobs from each subzone-$\iota$. Hence, when we evenly distribute jobs, some machine in subzone-$\iota_1$ may get jobs from subzone-$\iota_2$. As the subtypes of ${\veg}^{\iota_1}$ and ${\veg}^{\iota_2}$ are different, a job that fits a subzone-$\iota_2$ machine does not necessarily fit a subzone-$\iota_1$ machine. %The reader may wonder why we do not apply Lemma~\ref{lemma:color} by using the subzones instead of zones as colors. Briefly speaking, the purpose of introducing a sufficiently large $\sigma$ is to make sure that when we assign jobs in each group to a machine, say, $i$, the fact that $\max_{i,j}\|\eta_j^i\|_{\infty}$ and $m$ are bounded by some $\OFPT(1)$ can guarantee that any $m$ jobs fit if  

Note that megazone-$0$ only contains one zone (and one subzone). Thus all megazone-$0$ jobs (and thus megazone-$0$ jobs in each group), fit machine~$0$. From now on we only consider machine~$1$ to machine~$n$, and only consider groups of jobs which are not the last group.

Consider machines and jobs in each zone-$\nu$. Since in each zone ${\veg}^i$'s have the same type, we know if some coordinate, say, the $h$-th coordinate of ${\veg}^i$ is $0$, then the $h$-th coordinate of any zone-$\nu$ job is also $0$. Recall that we have set $\sigma\ge (m+\rho)\eta_{\max}$ to be sufficiently large such that if we add any $m$ %or $m'(\omega+1)\le (m+\rho)(\omega+1)$ 
jobs, the absolute value of each coordinate of the sum is no more than $\sigma$. %That is,  $\sigma = m\Delta \ge((s_D+5^{t_A}\varphi)\lambda^2_0)^{O(s_D+5^{t_A}\varphi)})$, where $\varphi=\Delta^{t_A+s_Dt_A}$ and $\lambda_0=(t_B\Delta)^{O(t_B\Delta^{t_B^2})}$. 
Hence, when we distribute jobs to machines in each zone-$\nu$, if the sum of $m$ jobs does not fit machine~$i$ (i.e., $\not\sqsubseteq{\veg}^i$), then the violation must occur at some coordinate of ${\veg}^i$ which is close-positive or close-negative (i.e., with a value in $[1,\sigma]\cup [-\sigma,-1]$). We call all close-positive or close-negative coordinates of each ${\veg}^i$ as critical coordinates. %By the definition of a zone, ${\veg}^i$'s share the same critical coordinates for all $i$'s in the same zone. 
Recall that $\veg^i$'s in the same zone share the same type, and hence the same critical coordinates. Let $CI_\nu=\{h_1^\nu,h_2^\nu,\cdots,h^\nu_{|CI_\nu|}\}$ be the set of critical coordinates for zone-$\nu$, that is, for any $i$ in zone-$\nu$, the $h_\ell^\nu$-th coordinate of ${\veg}^i$ falls in $[1,\sigma]\cup [-\sigma,-1]$. 

We consider the $h_\ell^\nu$-th coordinate of every job in zone-$\nu$. We say a job is {\it good} if its $h_\ell^\nu$-th coordinate is $0$ for {\it all} $1\le \ell\le |CI_\nu|$, and is {\it bad} otherwise (i.e., its $h_\ell^\nu$-th coordinate is nonzero for some $\ell$). It is clear that good jobs never cause trouble in the sense that any $m$ good jobs in zone-$\nu$ fit a zone-$\nu$ machine. It suffices to consider the scheduling of bad jobs.

Recall there are $\gamma_{\iota}$ slots (and hence $\gamma_{\iota}$ machines) in each subzone-$\iota$. We say a group is {\it bad} in subzone-$\iota$ if it contains more than $\gamma_{\iota}$ bad jobs in subzone-$\iota$, and is {\it good} if it is not a bad group in {\it any} subzone. We have the following lemmas regarding good and bad groups.

\begin{lemma}\label{lemma:good}
If a group is good and is not the last group in Lemma~\ref{obs:prop-ii}, then there is an assignment of jobs to machines such that the partial schedule $\veeta$ satisfies that $H\veeta=\ve0$, $\|\veeta\|_{\infty}\le m\eta_{\max}$ and $\veeta\sqsubseteq {\veg}$.
\end{lemma}
\begin{proof}
Notice that a good group does not necessarily contain exactly $m\gamma_\iota$ jobs in each subzone-$\iota$, but it contains no more than $\gamma_{\iota}$ bad jobs in each subzone-$\iota$. Hence, we reschedule jobs to obtain a partial schedule such that every machine in subzone-$\iota$ is assigned 1 or 0 bad job in subzone-$\iota$, together with $m-1$ or $m$ good jobs in zone-$\nu$ (that contains subzone-$\iota$). %We first claim that this is doable while respecting that each machine should receive $m$ tier-1 job and $m\omega$ tier-0 jobs: Indeed, we first assign one bad job to a machine, and this bad job can be tier-0 or tier-1. Then we simply assign remaining good jobs accordingly. 
We claim that, this partial schedule $\veeta$ satisfies Lemma~\ref{lemma:good}. First, by Lemma~\ref{obs:prop-ii}, jobs in every zone-$\nu$ is evenly distributed among machines in zone-$\nu$, %(indeed, we evenly distribute jobs in each subzone), 
and hence $H\veeta=\ve0$. Next, by the definition, a subzone-$\iota$ job is originally scheduled on a subzone-$\iota$ machine, and hence in the rescheduling it either stays at the original machine or moves to another subzone-$\iota$ machine. %Consider the capacity ${\veg}^i$ of each machine~$i$. 
By the definition of a subzone all machines in subzone-$\iota$ share the same value on critical coordinates. This means, a single bad job in subzone-$\iota$ fits any machine in subzone-$\iota$. Recall that the critical coordinate of a good job always has value $0$, so $m$ good jobs, or a bad job with $m-1$ good jobs fit any machine in subzone-$\iota$. Hence, $\veeta\sqsubseteq {\veg}$.  
\end{proof}

In the meantime, there are not too many bad groups as implied by the following lemma.
\begin{lemma}\label{lemma:bad}
The total number of bad groups is bounded by $(2\sigma+3)^{t_A}(1+5^{t_A}\varphi)\sigma t_A$.
\end{lemma}
\begin{proof}
Consider any slot $i$ in a subzone-$\iota$ contained in zone-$\nu$, 
and there are $|CI_\nu|$ critical coordinates. Let ${\veg}^i=({\veg}^i[1],{\veg}^i[2],\cdots,{\veg}^i[t_A])$. Recall there are $\gamma_\iota$ slots (indices) in subzone-$\iota$. Consider the summation of absolute value over critical coordinates of ${\veg}^i$'s in each subzone-$\iota$, we have
$$\sum_{i\in \textnormal{subzone}-\iota}\sum_{h\in CI_\nu}|{\veg}^i[h]|\le |CI_\nu|\sigma\gamma_{\iota}\le \sigma\gamma_{\iota} t_A.$$
Note that every bad job in subzone-$\iota$ lies in the same orthant with nonzero value at some critical coordinate, and must thus contribute at least $1$ to the above value. Recall that a bad group must be bad in at least one subzone, and any bad group in subzone-$\iota$ contains more than $\gamma_{\iota}$ bad jobs in subzone-$\iota$. Hence, a bad group in subzone-$\iota$ contributes at least $\gamma_{\iota}$ in total, which implies that there can be at most $\sigma t_A$ bad groups in subzone-$\iota$. Given that there are $(2\sigma+3)^{t_A}(1+5^{t_A}\varphi)$ subzones, there can be at most $(2\sigma+3)^{t_A}(1+5^{t_A}\varphi)\sigma t_A$ bad groups, and Lemma~\ref{lemma:bad} is proved. 
\end{proof}

%Finally, we are able to prove 

\subsubsection{Finalizing the proof of Lemma~\ref{lemma:sub}}\label{subsec:final-lemma}
%\noindent\textbf{\large Step 4: Finalizing the proof of Lemma~\ref{lemma:sub}}\label{subsec:final-lemma}

{By Lemma~\ref{obs:prop-ii}, except the last group, there are $\psi=\lfloor\frac{N-\rho}{m}\rfloor$ groups, where each group is either bad or good. By Lemma~\ref{lemma:bad} there are at most $(2\sigma+3)^{t_A}(1+5^{t_A}\varphi)\sigma t_A=(\Delta\eta_{\max})^{\Delta^{\OO(s_At_A+s_Dt_D)}}$ bad groups. Hence if $\frac{N-\rho}{m}\ge (2\sigma+3)^{t_A}(1+5^{t_A}\varphi)\sigma t_A+1$, %(or more precisely,   $N\ge((s_D+5^{t_A}\Delta^{t_A+s_Dt_A})(t_B\Delta)^{O(t_B\Delta^{t_B^2})})^{O(s_Dt_A+\Delta^{t_A+s_Dt_A} t_A5^{t_A})}\Delta^{t_A+s_Dt_A} t_A$ by plugging all the parameters), 
there will be at least one good group, and by Lemma~\ref{lemma:good} it induces some $\veeta$ such that ${\veeta}\sqsubseteq {\veg}$, $\|\veeta\|_{\infty}\le m\eta_{\max}\le \tau$ and $H\veeta=\ve0$. Further notice that only zone-0 jobs will be put on machine~$0$, and thus $\veeta^0$ is the summation of some $\veeta_j^0$'s. Therefore Lemma~\ref{lemma:sub} is proved.}

\subsection{Proof of Theorem~\ref{11-n}}\label{subsec:thm-proof}
Now we are ready to prove Theorem~\ref{11-n}. 
Consider an arbitrary $\veg\in \ker_{\Z}(H_{\textnormal{com}})$. As $\veg\in\ker_{\Z}(H_{\textnormal{com}}^{\textnormal{two-stage}})$, there exists a decomposition $\veg=\sum_j\ve\xi_j$ where $\ve\xi_j\in\ker_{\Z}(H_{\textnormal{com}}^{\textnormal{two-stage}})$, $\|\ve\xi_j\|_{\infty}=\OFPT(1)$ and $\ve\xi_j\sqsubseteq \veg$. But when can we guarantee that this can lead to a uniform decomposition? We observe that $B\ve\xi_j^0$'s are integers when $s_B=1$, and $B\ve\xi_j^0+A_i\ve\xi_j^i= 0$. If we aim for a uniform decomposition by merging $\ve\xi_j$'s, then the question becomes whether we can partition $\ve\xi_j$'s into different groups such that $B\ve\xi_j^0$'s within each group sum up to the same value (bounded by $\OFPT(1)$). An even partition does not need to exist, but we have the following sufficient condition. 

\begin{lemma}\label{lemma:number}
Let $x_1,x_2,\cdots,x_m\in \Z$ and $\zeta\in\Z_{>0}$ be integers such that $|x_i|\le \zeta$ for $i\in [m]$ and $\sum_{i=1}^m x_i=x$. If $x$ is a multiple of $(6\zeta^2+2\zeta+1)!$, then the $m$ integers can be partitioned into $m'$ subsets $T_1,T_2,\cdots,T_{m'}$ such that $\bigcup_{k=1}^{m'}T_k=[m]$, and for all $k\in [m']$ it holds that $|T_k|\le 2^{\OO(\zeta^2\log\zeta)}$, $\sum_{i\in T_k}x_i\in \{0,sgn(x)\cdot (6\zeta^2+2\zeta+1)!\}$ where $sgn$ denotes the standard sign function such that $sgn(x)=1$ if $x>0$, $sgn(x)=-1$ if $x<0$, and $sgn(x)=0$ if $x=0$.
\end{lemma}
%\begin{lemma}\label{lemma:number}
%Let $x_1,x_2,\cdots,x_m\in \Z_{> 0}$ with $x_i\le \zeta$ for $i\in [m]$ and $\sum_{i=1}^m x_i=x$. If $x$ is a multiple of $(\zeta+1)!$, then the $m$ integers can be partitioned into $m'=x/(\zeta+1)!$ subsets $T_1,T_2,\cdots,T_{m'}$ such that $\bigcup_{k=1}^{m'}T_k=[m]$, and $\sum_{i\in T_k}x_i=(\zeta+1)!$ for all $k\in [m']$.
%\end{lemma}

%\subsection{Uniform condition and the auxiliary Lemma -- Lemma~\ref{lemma:sub}}\label{subsec:lambda}
%where $\varphi=\OO_{FPT}(1)$. Notice that $A_i$'s and $D_i$'s above are not necessarily distinct. 

With Lemma~\ref{lemma:number}, we are able to prove the following.

\begin{lemma}\label{lemma:dec-1}
Let $\veg\in \ker_{\Z}(H_{\textnormal{com}})$. Let 
\begin{eqnarray*}
\lambda=(6\lambda_0^2+2\lambda_0+1)!=2^{2^{2^{\OO(t_B^2\log \Delta)}}}, \textnormal{ where } \lambda_0:=\Delta t_B g_{\infty}(H_{\textnormal{com}}^{\textnormal{two-stage}})=2^{2^{\OO(t_B^2\log \Delta)}}.
\end{eqnarray*}
If $B\veg^0$ is a multiple of $\lambda$, then $\veg$ admits a uniform decomposition $\veg=\sum_{j=1}^N\veeta_j$ such that $\|\veeta_j\|_{\infty}\le 2^{2^{2^{\OO(t_B^2\log \Delta)}}}$. Furthermore, $B\veeta_j^0$ is a multiple of $\lambda$ for all $j$.
\end{lemma}

Now we are ready to prove our main theorem.
\begin{proof}[Proof of Theorem~\ref{11-n}]
Consider any $\veg\in \ker_{\Z}(H_{\textnormal{com}})$. Clearly $B(\lambda\veg^0)$ is a multiple of $\lambda$, thus by Lemma~\ref{lemma:dec-1}, $\lambda\veg$ admits a uniform decomposition $\lambda\veg = \sum_{j=1}^N\veeta_j$ where {$\|\veeta_j\|_{\infty}\le \eta_{\max}=2^{2^{2^{\OO(t_B^2\log \Delta)}}}$ and every $B\veeta_j^0$ is a multiple of $\lambda$. } 

If this decomposition is not $\omega$-balanced for $\omega\le (\Delta t_D\eta_{\max})^{\OO(s_D^2)}$, then by Lemma~\ref{lemma:balance} we obtain $\veeta\sqsubseteq \lambda\veg$ with $\|\veeta\|_{\infty}\le (\Delta t_D\eta_{\max})^{\OO(s_D^2)}$, $\veeta\in \ker_{\Z}(H_{\textnormal{com}})$ and $B(\lambda\veg^0-\veeta^0)=\ve 0$. $B\veeta^0$ is a multiple of $\lambda$.
Otherwise this decomposition is $\omega$-balanced. By Lemma~\ref{lemma:all-balance}, we can obtain a uniform decomposition $\lambda\veg = \sum_{j=1}^{N'}{\veeta}_j'$ such that $\max_j\|{\veeta}_j'\|\le \omega\eta_{\max}$ and all $\veeta_j'$'s are tier-1.  According to Lemma~\ref{lemma:sub},
if $\lambda\|\veg\|_{\infty}>\tau$ for $\tau=(\omega\Delta\eta_{\max})^{\Delta^{\OO(s_At_A+s_Dt_D)}}=2^{2^{\OO(s_At_A\log\Delta+s_Dt_D\log\Delta)}\cdot 2^{2^{\OO(t_B^2\log\Delta)}}}$, then we are able to find some $\veeta\sqsubseteq \lambda\veg$ such that $H_{\textnormal{com}}\veeta=\ve 0$, $\|\veeta\|_\infty=\OFPT(1)$ and $\veeta^0=\sum_{j\in S}\veeta_j^0$ for some $S\subseteq [N]$. As every $B\veeta_j^0$ is a multiple of $\lambda$, $B\veeta^0$ is also a multiple of $\lambda$. In both cases, we find $\veeta\sqsubseteq \lambda\veg$ where $B\veeta^0$ is a multiple of $\lambda$.

Now consider $\lambda\veg-\veeta$. Obviously $\lambda\veg-\veeta\in \ker_{\Z}(H_{\textnormal{com}})$. It is easy to see $B(\lambda\veg^0-\veeta^0)$ is a multiple of $\lambda$. Thus, if $\|\lambda\veg-\veeta\|_{\infty}> \tau$ we can continue to decompose $\lambda\veg-\veeta$ using our argument above. Observing that $s_A=1$, Theorem~\ref{11-n} is proved. 
%{\color{red}If $\|\lambda\veg-\veeta\|_{\infty}$ is } too large, then we can continue to find out $\veeta'\sqsubseteq \lambda\veg-\veeta$ such that $\|\veeta'\|_\infty=\OFPT(1)$. Iteratively applying the above argument, Theorem~\ref{11-n} is proved.  
\end{proof}
\begin{remark*} Theorem~\ref{11-n} is also true for almost combinatorial 4-block $n$-fold IP. Now $B$ is not a $1\times t_B$ matrix, but rather an $s_B\times t_B$ matrix with rank 1. For such a matrix $B$, we can always transform it into $\bar{B}$, in which the first row is $\ver_1^{\top}$, and all the other rows are $\ve0$. It implies that when $\text{rank}(B)=1$, it is sufficient to consider such a case $B=(\ver_1,\ve 0,\ldots,\ve 0)^{\top}$, where $\ver_1\neq \ve 0$. Then we observe that for any $\vex\in\Z^{t_B}$, $B\vex=(\ver_1\cdot\vex,\ve0,\cdots,\ve0)$. Hence, our argument in the proof above applies directly, i.e., Theorem~\ref{11-n} holds for almost combinatorial 4-block $n$-fold IP (see Appendix~\ref{appsec:almost} for a formal proof). In other words, Theorem~\ref{11-n} and our FPT algorithm for combinatorial 4-block $n$-fold IP  remain true for almost combinatorial 4-block $n$-fold IP. Such a generalization allows submatrices $A_i$'s to contain multiple rows subject to that these rows are  ``local constraints''. 
\end{remark*}
%Having the above theorem, the decomposability of $\lambda\veg$ for arbitrary basis $\veg$ inspires us to consider the property of a Graver basis. When $\veg$ is a Graver basis, there exists an integer $\lambda$ satisfying the conclusion in Theorem~\ref{11-n}. What's more, we have the following corollary and in Subsection~\ref{sec1} we will solve 4-block $n$-fold IP by using this corollary.

\section{Algorithms}\label{sec:alg}
Using Theorem~\ref{11-n}, we are able to bound the $\ell_\infty$-norm of the Graver basis elements:

\begin{theorem}\label{coro:graver}
Let $\veg\in\G(H_{\textnormal{com}})$ be a Graver basis element, then $\|\ve g\|_\infty= g_{\infty}(H_{\textnormal{com}})$ where $g_{\infty}(H_{\textnormal{com}})\le 2^{2^{\OO(t_A\log\Delta+s_Dt_D\log\Delta)}\cdot 2^{2^{\OO(t_B^2\log\Delta)}}}\cdot n=\OFPT(n).$ 
\end{theorem}

Utilizing Theorem~\ref{coro:graver} and the iterative augmentation framework (see Section~\ref{sec:pre}), we are able to prove the following theorem.

\begin{theorem}\label{thm20}
 Consider combinatorial 4-block $n$-fold IP with a separable convex objective function $f$ mapping $\Z^{t_B+nt_A}$ to $\Z$. Let $P$ be the set of feasible integral points, and let $\hat{f}:=\max_{x,y\in P}(f(x)-f(y))$. Then it can be solved in
$2^{2^{\OO(t_A\log\Delta+s_Dt_D\log\Delta)}\cdot 2^{2^{\OO(t_B^2\log\Delta)}}}\cdot n^{4}\hat{L}^2\log^2(\hat{f})=\OFPT(n^4\hat{L}^2 \log^2(\hat{f}))$ time, where $\hat{L}$ denotes the logarithm of the largest number occurring in the input.
\end{theorem}
The running time can be improved to $\OFPT(n^{5+o(1)})$ if the objective function is linear. See Appendix~\ref{linear_obj} for the proof.

\section{Applications in Scheduling with High Multiplicity}\label{appli}
It has been shown by Knop and Kouteck{\`y}~\cite{knop2018scheduling} that the classical scheduling problems $R||C_{\max}$ and $R||\sum_{\ell}w_\ell C_\ell$ can be modeled as $n$-fold IPs, based on which FPT algorithms can be developed. However, when we try to model more sophisticated scheduling problems, especially scheduling with rejection $R||C_{\max}+E$ or bicriteria scheduling $R||\theta C_{\max}+\sum_{\ell}w_\ell C_\ell$, we run into 4-block $n$-fold IP. This is because for these problems, $C_{\max}$ needs to be taken as a variable in the IP, while for $R||C_{\max}$ we can use binary search on $C_{\max}$ and hence $n$-fold IP is sufficient.

%We consider two variants for this application of scheduling. One is the bi-objective of makespan and rejection cost, the other is the bi-objective  of makespan and the total weighted completion time. 

We formally describe the scheduling problem. Given are $m$ machines and $k$ different types of jobs, with $N_j$ jobs of type $j$. A job of type $j$ has a processing time of $p^i_j\in\Z_{\ge 0}$ if it is processed by machine~$i$. 

For scheduling with rejection $R||C_{\max}+E$, every job of type $j$ also has a rejection cost $u_j$. A job is either processed on one of the machine, or is rejected. The goal is to minimize the makespan $C_{\max}$ plus the total rejection cost $E$. 

\begin{theorem}\label{thmmm}
$R||C_{\max}+E$ can be solved in
 $m^{5+o(1)} 2^{2^{\OO(k^2\log p_{\max})}\cdot 2^{2^{\OO(\log p_{\max})}}}+|I|$ time, where $|I|$ denotes the length of the input.
\end{theorem}
 More precisely, $|I|$ is bounded by $\OO(kp_{\max} (\max\{\log N_{\max}, \log u_{\max}\}))$ where $N_{\max}=\max_j N_j$ and $u_{\max}=\max_j u_j$. One may suspect that the problem can be solved through the generalized $n$-fold IP by guessing out the value of $C_{\max}$. However, this will require $p_{\max}\cdot \max_jN_j$ enumerations. See 
 Appendix~\ref{sche1} for a detailed proof of Theorem~\ref{thmmm}. 
 
For bicriteria scheduling $R||\theta C_{\max}+\sum_{\ell}w_\ell C_\ell$, each job $\ell$ of type $j$ has a weight $w_j$, and the goal is to find an assignment of jobs to machines such that $\theta C_{\max}+\sum_{\ell}w_\ell C_\ell$ is minimized, where $C_\ell$ is the completion time of job $\ell$, and $\theta$ is a fixed input value.

\begin{theorem}\label{them18}
 $R||\theta C_{\max}+\sum_{\ell}w_\ell C_\ell$ can be solved in
 $m^4 2^{2^{\OO(k^2\log p_{\max})}\cdot 2^{2^{\OO(\log p_{\max})}}}|I|^4$ time, where $|I|$ denotes the length of the input.
\end{theorem}
More precisely, $|I|$ is bounded by $\OO(kp_{\max} (\max\{\log N_{\max},\log w_{\max}\}))$ where $N_{\max}=\max_j N_j$, $w_{\max}=\max_j w_j$. See 
 Appendix~\ref{sche2} for a detailed proof of Theorem~\ref{them18}. 

For identical machines, $k\le p_{\max}$ and we obtain FPT algorithms parameterized by $p_{\max}$.

\clearpage
\bibliographystyle{plain}
\bibliography{B_block}

\clearpage

\appendix

\section{Omitted contents in Section~\ref{sec:structure}}
\subsection{Proof of Theorem~\ref{thm:lower-bound}}\label{appsec:lower-bound}
\begin{T3}
 There exists a 4-block $n$-fold IP where $s_B=s_D=1$ such that $\|\veg\|_{\infty}=\Omega(n)$ for some Graver basis element $\veg$.
\end{T3}
\begin{proof}
Consider the $4$-block $n$-fold IP where its constraint matrix is defined by $H_0$ in which $C=(-1,-1,-1)$, $D=(5,3)$, $B=(0,-1,1)$ and $A=(3,4)$. Let $\veg=(\veg^0,\veg^1,\cdots,\veg^n)$ where
$\veg^0=(1,n-1,n)$, and $\veg^1=\veg^2=\cdots=\veg^n=(1,-1)$. 

It is easy to see that $\|\ve g\|_\infty=n$. Meanwhile we have the following:
\begin{eqnarray}
	&&C\veg^0+ D \sum_{i=1}^{n}\veg^i=(-1,-1,-1)\cdot(1,n-1,n)+(5,3)\cdot(n,-n)=0,\nonumber\\	
	&&B\veg^0+ A\veg^i=(0,-1,1)\cdot(1,n-1,n)+(3,4)\cdot(1,-1)=0, \quad\forall 1\le i\le n,\nonumber
\end{eqnarray}
which means that $H_0\veg=\ve0$. 
In what follows, we prove that $\veg$ is a Graver basis element, i.e., there does not exist any non-zero $\veeta\sqsubseteq \veg$ such that $H_0\veeta=\ve0$. %In other words, if there is a non-zero vector $\veeta\sqsubseteq \veg$  realizing $H_{\textnormal{com}}\veeta=\ve0$, then we show that $\veeta= \veg$.
Towards this, we assume on the contrary that there exists a vector $\veeta\sqsubseteq \veg$  such that $H_0\veeta=\ve0$. Consequently, we have 
\begin{eqnarray}
	&& C\veeta^0+ D \sum_{i=1}^{n}\veeta^i=\ve 0,\label{eq29}\\
 &&B\veeta^0+ A\veeta^i=\ve 0,\quad \forall 1\le i\le n. \label{eq28}
\end{eqnarray}
 
 Let $\veeta^0=(a,b,c)$ and $\veeta^i=(a_i,b_i)$. We first make the following claim.
 \begin{claim}
     $B\veeta^0\neq 0$.
 \end{claim}
 \begin{proof}
  Suppose on the contrary that $B\veeta^0=\ve 0$. Given that $\veeta^i=(a_i,b_i)\sqsubseteq (1,-1)$, we get $a_i\in\{0,1\}$ and $b_i\in\{0,-1\}$. By~\eqref{eq28}, $A\veeta^i=3a_i+4b_i=0$. Consequently, we have $a_i=b_i=0$. Plug $\veeta^i=(0,0)$ into the Eq~\eqref{eq29}, we get $C\veeta^0=\ve 0$. That is, $(-1,-1,-1)\cdot \veeta^0=-a-b-c =0$. Since $\veeta^0\sqsubseteq (1,n-1,n)$, we get $a,b,c\ge 0$, implying that $a=b=c=0$. Hence, $\veeta=\ve0$, which is a contradiction. Thus, $B\veeta^0\neq0$ and the claim follows.
 \end{proof}

Recall that $a_i\in\{0,1\},b_i\in\{0,-1\}$ for every $i$. It is easy to see that there are three possibilities regarding the value of $B\veeta^0$: i) $B\veeta^0=4$ if $a_i=0$, $b_i=-1$ for some $i$. ii) $B\veeta^0=-3$ if $a_i=1$, $b_i=0$ for some $i$. iii) $B\veeta^0=1$ if $a_i=1$, $b_i=-1$ for some $i$. Consequently, in each case all $a_i$'s ($b_i$'s) must take the same value, i.e., there are three possibilities regarding the values of $a_i$'s and $b_i$'s:
% Given that $B\veeta^0\neq0$ and $A\veeta^i=3a_i+4b_i=-B\veeta^0$, where $a_i\in\{0,1\},b_i\in\{0,-1\}$, we know there are three possibilities regarding the values of $a_i$'s and $b_i$'s: as follows: 
 \begin{itemize}
 	\item $a_i=0$, $b_i=-1$ for all $i$. Then we have $C\veeta^0+ D \sum_{i=1}^{n}\veeta^i=(-1,-1,-1)\cdot \veeta^0+(5,3)\cdot (0,-n)=-(a+b+c)-3n$. Since $\veeta^0\sqsubseteq (1,n-1,n)$, $a,b,c\ge 0$, whereas $C\veeta^0+ D \sum_{i=1}^{n}\veeta^i<0$, contradicting Eq~\eqref{eq29}. 
 	
 	\item $a_i=1$, $b_i=0$ for all $i$. Then we have $C\veeta^0+ D \sum_{i=1}^{n}\veeta^i=(-1,-1,-1)\cdot \veeta^0+(5,3)\cdot (n,0)=-(a+b+c)+5n$. Using that $\veeta^0\sqsubseteq (1,n-1,n)$, $a\le 1$, $b\le n-1$ and $c\le n$, whereas $C\veeta^0+ D \sum_{i=1}^{n}\veeta^i\ge -2n+5n> 0$, contradicting Eq~\eqref{eq29}. 
 	
    \item $a_i=1$, $b_i=-1$ for all $i$. Then we have $C\veeta^0+ D \sum_{i=1}^{n}\veeta^i=(-1,-1,-1)\cdot \veeta^0+(5,3)\cdot (n,-n)=-(a+b+c)+2n$. Using that $a\in [0,1], b\in [0,n-1]$ and $c\in [0,n]$, $-(a+b+c)+2n=0$ if and only if $a=1,b=n-1$ and $c=n$. Hence, $\veeta^0= (1,n-1,n)$.
 \end{itemize}
 
 The above argument implies that $\veeta=(\veeta^0,\veeta^1,\cdots,\veeta^n)$ where $\veeta^0=(1,n-1,n)$ and $\veeta^i=(1,-1)$ for all $i$, and thus $\veeta=\veg$, implying that $\veg$ is a Graver basis element. Hence, Theorem~\ref{thm:lower-bound} is proved. 
\end{proof}

\subsubsection{Proof of Lemma~\ref{lemma:balance}}\label{appsec:balance}
\begin{T4}
For any $\bar{\veg}\in \ker_{\Z}(H)$, if $\bar{\veg}$ admits a uniform decomposition $\bar{\veg}=\sum_{j=1}^{\bar{N}}\bar{\veeta}_j$ where $\|\bar{\veeta}_j\|_{\infty}\le \bar{\eta}_{\max}$, then there exists ${\veg}\in \ker_{\Z}(H)$ such that ${\veg}\sqsubseteq \bar{\veg}$, $B_i(\bar{\veg}^0-\veg^0)=\ve 0$ for all $i\in [n]$, and ${\veg}$ admits an $\omega$-balanced uniform decomposition for $\omega\le (\Delta t_D \bar{\eta}_{\max})^{\OO(s_D^2)}$. Moreover, if $\bar{\veg}-\veg\neq 0$, then we have $\bar{\veg}-\veg=\veg_1+\veg_2+\cdots+\veg_p$ for some $p\in \mathbb{Z}$ and $\veg_j\in \ker_{\Z}(H)$, and furthermore, $\veg_j\sqsubseteq \bar{\veg}-\veg$ and $\|\veg_j\|_\infty\le (\Delta t_D \bar{\eta}_{\max})^{\OO(s_D^2)}$.
\end{T4}

%\begin{lemma}\label{lemma:balance}
%For any $\bar{\veg}\in \ker_{\Z}(H)$, if $\bar{\veg}$ admits a uniform decomposition $\bar{\veg}=\sum_{j=1}^{\bar{N}}\bar{\veeta}_j$, then there exists ${\veg}\in \ker_{\Z}(H)$ such that ${\veg}\sqsubseteq \bar{\veg}$, and ${\veg}$ admits an $\omega$-balanced uniform decomposition for $\omega\le $. Moreover, if $\bar{\veg}-\veg\neq 0$, then we have $\bar{\veg}-\veg=\veg_1+\veg_2+\cdots+\veg_p$, where $p\in \mathbb{Z}$, and furthermore, $\veg_j\in \ker_{\Z}(H_{\textnormal{com}})$, $\veg_j\sqsubseteq \lambda\veg$ and $\|\veg_j\|_\infty=2^{2^{\OO(t_B^2\log\Delta+s_Dt_A\log\Delta)}}$..
%\end{lemma}
\begin{proof}
%Suppose $\bar{N}_0>\omega \bar{N}_1$. 
For simplicity let $D_0=C$ and we consider the equation below:
\begin{eqnarray}\label{eq:balance}
0=(D_0,D_1,\cdots,D_n)\sum_{j=1}^{\bar{N}}\bar{\veeta}_j=\sum_{i=0}^n\sum_{j=1}^{\bar{N}}D_i\bar{\veeta}_j^i.
\end{eqnarray}
Obviously each summand on the right side of Eq~\eqref{eq:balance} is an $s_D$-dimensional vector such that $\|D_i\bar{\veeta}_j^i\|_{\infty}\le \Delta t_D \bar{\eta}_{\max}$. We say $D_i\bar{\veeta}_j^i$ is a tier-0 (or tier-1) summand if $\bar{\veeta}_j$ is tier-0 (or tier-1). Consequently, there are $(n+1)\bar{N}_0$ tier-0 summands and $(n+1)\bar{N}_1$ tier-1 summands. According to Lemma~\ref{lemma:merging-lemma}, all the summands can be divided into $m'$ subsets $T_1$, $T_2$, $\cdots$, $T_{m'}$ such that each subset contains at most $u:=(\Delta t_D \bar{\eta}_{\max})^{\OO(s_D^2)}$ summands, and consequently $m'\ge {(n+1)\bar{N}}/u$. It is easy to see that if ${(n+1)\bar{N}}/u>(n+1)\bar{N}_1$ (or equivalently, $\bar{N}_0>(u-1)\bar{N}_1$), then $m'>(n+1)\bar{N}_1$, and by Pigeonhole principle there exists some $T_k$ such that $T_k$ does not contain any tier-1 summand. Consider such $T_k$ and let it contain summands $D_{i_1}\bar{\veeta}_{j_1}^{i_1}$ to $D_{i_u}\bar{\veeta}_{j_u}^{i_u}$ where every $\veeta_{j_\ell}$ is tier-0. 

Now we let $\veeta=(\veeta^0,\cdots,\veeta^n)$ be such that $\veeta^{i}=\sum_{\ell:i_\ell=i}\bar{\veeta}_{j_\ell}^{i_\ell}$ (specifically, ${\veeta}^i=\ve 0$ if $i_\ell\neq i$ for all $\ell$). Then it follows directly that $(D_0,D_1,\cdots,D_n)\veeta=\ve 0$. Furthermore, by the definition of tier-0, for $i_\ell=0$ we have $B_i\bar{\veeta}_{j_\ell}^0=\ve 0$ for all $i\in [n]$, and for any $i_\ell\ge 1$ we have $A_{i_\ell}\bar{\veeta}_j^{i_\ell}=\ve 0$. Hence, $H^{\textnormal{two-stage}}\veeta=\ve 0$. Consequently, $H\veeta=\ve 0$ and $\veeta\sqsubseteq \bar{\veg}$. As $\veeta$ consists of at most $u$ bricks, $\|\veeta\|_{\infty}\le u\eta_{\max}=(\Delta t_D \bar{\eta}_{\max})^{\OO(s_D^2)}$.

To summarize, as long as $\bar{N}_0>(u-1)\bar{N}_1$ for $u=(\Delta t_D \bar{\eta}_{\max})^{\OO(s_D^2)}$ we can find $\veeta\in\ker_{\Z}(H)$ satisfying that $H\veeta=\ve 0$, $\veeta\sqsubseteq \bar{\veg}$ and $\|\veeta\|_{\infty}\le u\eta_{\max}=(\Delta t_D \bar{\eta}_{\max})^{\OO(s_D^2)}$. Hence, we can iteratively apply our argument above to decompose $\bar{\veg}$ until it becomes $(u-1)$-balanced, and Lemma~\ref{lemma:balance} is proved.
\end{proof}

\subsection{Proof of Lemma~\ref{lemma:all-balance}}
\begin{T11}
If ${\veg}=\sum_{j=1}^{\bar{N}}\bar{\veeta}_j$ is an $\omega$-balanced uniform decomposition where $\bar{\eta}_{\max}=\max_{j\in[\bar{N}]}\|\bar{\veeta}_j\|_{\infty}$, then ${\veg}$ admits a uniform decomposition ${\veg}=\sum_{j=1}^{N}{\veeta}_j$ such that every $\veeta_j$ is tier-1, and ${\eta}_{\max}=\max_{j\in[{N}]}\|{\veeta}_j\|_{\infty}\le \omega\bar{\eta}_{\max}$.
\end{T11}
\begin{proof}
Consider the $\omega$-balanced uniform decomposition ${\veg}=\sum_{j=1}^{\bar{N}}\bar{\veeta}_j$ and suppose $(\omega'-1)N_1\le N_0\le \omega'N_1$ for $1\le \omega'\le \omega$. Then we pick $N_0-(\omega'-1)N_1\le N_1$ tier-0 vectors, and merge each of them with a distinct tier-1 vector. By doing so we obtain an exact $(\omega'-1)$-balanced uniform decomposition. Next, we merge each tier-1 vector with exactly $\omega'-1$ distinct tier-0 vectors. Then we obtain a uniform decomposition with only tier-1 vectors. It is easy to see that at most $\omega'\le \omega$ vectors are merged together, and thus the infinity norm increases by at most $\omega$ times. Hence, Lemma~\ref{lemma:all-balance} is true.
\end{proof}

\subsection{A figure in Section~\ref{subsec:notion}}\label{figure}
\begin{figure}[h]
	\centering
	\includegraphics[width=0.95\linewidth]{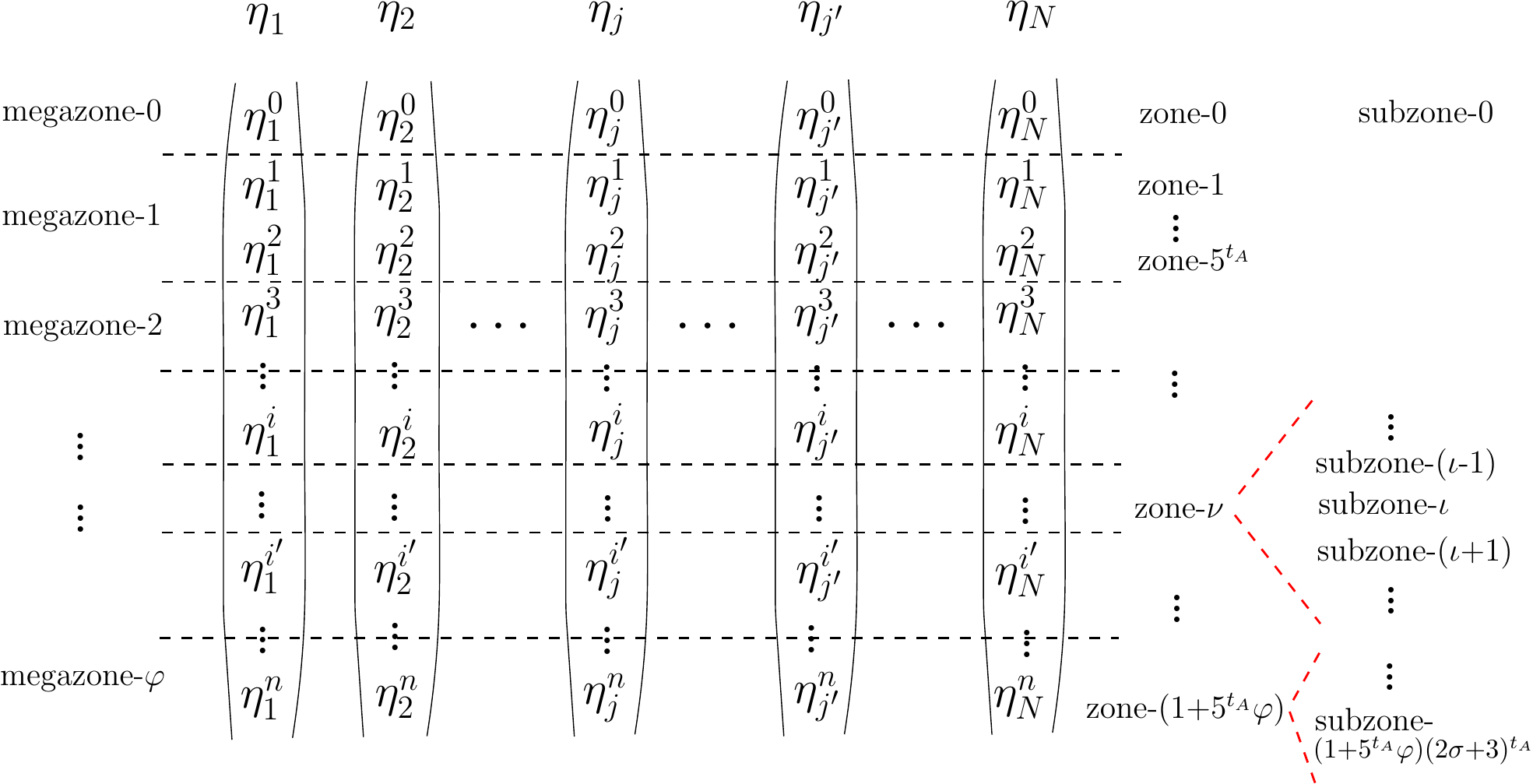}
	\caption{The relationships among megazones, zones and subzones.}
	\label{fig1}
\end{figure}

Notice that two adjacent bricks in the same megazone in a column are not necessarily belonging to the same zone. Two adjacent bricks in the same zone in a column are not necessarily belonging to the same subzone.

\subsection{Proof of Lemma~\ref{lemma:color}}\label{appsec:color}
\begin{T5}
Let $\vex_1,\ldots,\vex_{M}\in \mathbb{Z}^d$ be a sequence of vectors such that $\|\vex_i\|_{\infty}\le \zeta$ for some $\zeta\ge 1$ and every $i=1,\ldots,M$. Furthermore, there are $\mu$ colors, and each vector $\vex_i$ is associated with one color. There are in total $\alpha_j\overline{m}$ vectors of color $j$ where $\alpha_j,\overline{m}\in\Z_{>0}$ and $\sum_{j=1}^{\mu}\alpha_j=\alpha$, $M=\alpha \overline{m}$. Supposing that $\sum_{i=1}^{M}\vex_i=\ve0$ and $M$ is sufficiently large (i.e., $M>(2d\zeta+2\mu\zeta+1)^{d+\mu}\alpha+\alpha+d+\mu$), then among $\vex_1,\cdots,\vex_{M}$ we can find $\alpha_jm$ vectors of each color $j$ such that their summation is $\ve0$, and $m\le (2d\zeta+2\mu\zeta+1)^{d+\mu}$. 
\end{T5}
\begin{proof}
We lift the vectors in $\Z^d$ to $\Z^{d+\mu}$ such that if $\vex_i$ is of color $j$, then it is mapped to $\vey_i=(\vex_i,\vece_j)$ where $\vece_j=(\underbrace{0,0,\cdots,0}_{j-1},1,\underbrace{0,\cdots,0}_{\mu-j})$ is the vector with its $j$-th coordinate being $1$. Given that there are $\alpha_j\overline{m}$ vectors of color $j$, we have:
$$\sum_{i=1}^{M}\vey_i=(\underbrace{0,0,\cdots,0}_{d},\alpha_1\overline{m},\alpha_2\overline{m},\cdots,\alpha_{\mu}\overline{m}).$$
Denote by $\vey$ the right side of the above equation. Note that $\|\ve y_i\|_{\infty}\le \zeta$. Applying the Steinitz Lemma, then there exists a permutation $\pi$ such that for every $\ell\in [M]$ we have
$$\|\sum_{i=1}^\ell \vey_{\pi(i)}-\frac{\ell-d-\mu}{M}\vey\|_{\infty}\le (d+\mu)\zeta.$$
Notice that $M=\alpha \overline{m}$, and we have 
\begin{eqnarray}\label{eq:ex-1}
\frac{\ell-d-\mu}{M}\vey=(\ell-d-\mu)\cdot(\underbrace{0,0,\cdots,0}_{d},\frac{\alpha_1}{\alpha},\frac{\alpha_2}{\alpha},\cdots,\frac{\alpha_{\mu}}{\alpha}),
\end{eqnarray}
and note that if $\ell-d-\mu$ is a multiple of $\alpha$, then the right side is an integral vector. Consider $\ell_k=k\alpha+d+\mu$ for $k\in\Z_{>0}$, so by Eq~\eqref{eq:ex-1} $\vez_k:=\sum_{i=1}^{\ell_k} \vey_{\pi(i)}-\frac{\ell_k-d-\mu}{M}\vey$ is a $(d+\mu)$-dimensional integral vector whose $\ell_{\infty}$-norm is bounded by $(d+\mu)\zeta$. Hence, there are at most $(2d\zeta+2\mu\zeta+1)^{d+\mu}$ distinct $\vez_k$'s, which implies that if $M$ is large enough (and thus induces sufficiently many $\vez_k$'s), %larger than the maximum possible value of $\ell_k$,
i.e., $M>((2d\zeta+2\mu\zeta+1)^{d+\mu}+1)\alpha+d+\mu$, then there must exist two integers $k_1,k_2\le (2d\zeta+2\mu\zeta+1)^{d+\mu}+1$ such that $\vez_{k_1}=\vez_{k_2}$, and consequently 
$$\sum_{i=k_1\alpha+d+\mu+1}^{k_2\alpha+d+\mu} \vey_{\pi(i)}=\frac{(k_2-k_1)\alpha}{M}\vey=(k_2-k_1)\cdot(\underbrace{0,0,\cdots,0}_{d},\alpha_1,\alpha_2,\cdots,\alpha_{\mu}).$$
This means, we have found a subset of $\vex_i$'s with at most $(2d\zeta+2\mu\zeta+1)^{d+\mu}\alpha$ vectors which add up to $\ve0$, and furthermore, the total number of vectors of each color $j$ is proportional to $\alpha_j$.   
\end{proof}

\subsection{Proof of Lemma~\ref{lemma:number}}
To prove Lemma~\ref{lemma:number}, we need the following lemma.

\begin{lemma}\label{lemma:number-1}
Let $x_1,x_2,\cdots,x_m\in \Z_{> 0}$ with $x_i\le \zeta$ for $i\in [m]$ and $\sum_{i=1}^m x_i=x$. If $x$ is a multiple of $(\zeta+1)!$, then the $m$ integers can be partitioned into $m'=x/(\zeta+1)!$ subsets $T_1,T_2,\cdots,T_{m'}$ such that $\bigcup_{k=1}^{m'}T_k=[m]$, and $\sum_{i\in T_k}x_i=(\zeta+1)!$ for all $k\in [m']$.
\end{lemma}
\begin{proof}

Since $x_i$'s can only take at most $\zeta$ distinct values, we let $u_j$ be the total number of $x_i$'s taking the value $j\le \zeta$. We can divide the $u_j$ numbers into $\lceil \frac{ju_j}{\zeta!}\rceil$ groups, with all except 1 group containing $\frac{\zeta!}{j}$ numbers, and one group containing $res_j=u_j-\frac{\zeta!}{j}\lfloor \frac{ju_j}{\zeta!}\rfloor$ numbers. Consequently, we obtain a grouping of $x_i$'s such that there are $\zeta$ groups where the summation of numbers inside is $j\cdot res_j<\zeta!$, together with $h$ other groups where the summation of numbers inside any group is exactly $\zeta!$. Notice that 
$$x=\sum_{j=1}^{\zeta}j\cdot res_j+\zeta!\cdot h,$$
and $x$ is a multiple of $\zeta!$, hence $\sum_{j=1}^{\zeta}j\cdot res_j\le \zeta\cdot\zeta!$ is also a multiple of $\zeta!$, and we let it be $a\zeta!$ for $a\le \zeta$. Agglomerating these $\zeta$ groups, we obtain $h+1$ groups, where the summation of numbers within 1 group (called extra group) is $a\zeta!$ and the summation of numbers within any other group (called regular group) is exactly $\zeta!$. Given that $x=a\zeta!+h\zeta!$ is a multiple of $(\zeta+1)!$, we can further agglomerate the extra group with $\zeta+1-a$ regular groups, that is, $x=[a\zeta!+(\zeta+1-a)\zeta!]+(h-\zeta-1+a)\zeta!$, and then the remaining regular groups are evenly divided into subsets such that each subset contains $\zeta+1$ regular groups. This is possible since $x-(\zeta+1)!$ is a multiple of $(\zeta+1)!$. It is easy to see that now the numbers within every agglomerated group sum up to $(\zeta+1)!$, and Lemma~\ref{lemma:number-1} is proved. 
\end{proof}

Now we are ready to prove Lemma~\ref{lemma:number}.
\begin{T6}
Let $x_1,x_2,\cdots,x_m\in \Z$ and $\zeta\in\Z_{>0}$ be integers such that $|x_i|\le \zeta$ for $i\in [m]$ and $\sum_{i=1}^m x_i=x$. If $x$ is a multiple of $(6\zeta^2+2\zeta+1)!$, then the $m$ integers can be partitioned into $m'$ subsets $T_1,T_2,\cdots,T_{m'}$ such that $\bigcup_{k=1}^{m'}T_k=[m]$, and for all $k\in [m']$ it holds that $|T_k|\le 2^{\OO(\zeta^2\log\zeta)}$, $\sum_{i\in T_k}x_i\in \{0,sgn(x)\cdot (6\zeta^2+2\zeta+1)!\}$ where $sgn$ denotes the standard sign function such that $sgn(x)=1$ if $x>0$, $sgn(x)=-1$ if $x<0$, and $sgn(x)=0$ if $x=0$.
\end{T6}
\begin{proof}
Without loss of generality, we assume $x\ge 0$ (If $x<0$, we simply apply the argument below to the sequence of $-x_i$'s).
Notice that $x_i$'s do not necessarily lie in the same orthant.  We first apply Lemma~\ref{lemma:merging-lemma} to the sequence of $x_i$'s, and obtain a partition of $[m]$ into $m_1$ subsets $T_1',T_2',\cdots,T_{m_1}'$ such that for every $k\in [m_1]$, $\sum_{i\in T_j'}x_i\sqsubseteq x$ and $|T_j'|\le 6\zeta+2$. Let $y_j=\sum_{i\in T_j'}x_i$. If $x=0$, then $y_j=0$ for all $j$ and Lemma~\ref{lemma:number} is proved. Otherwise $x>0$, and it follows that $y_j\ge 0$ for all $j$. Consider all $y_j$'s which are positive. Without loss of generality, let them be $y_1,y_2,\cdots,y_{m_2}$. We know that $y_j>0$ for $j\in [m_2]$, $y_j\le \zeta(6\zeta+2)$ and $\sum_j y_j=x$ where $x$ is a multiple of $(6\zeta^2+2\zeta+1)!$. Applying Lemma~\ref{lemma:number-1}, we can obtain a partition of $[m_2]$ into $m_3$ subsets $T_1'',T_2'',\cdots,T_{m_3}''$ such that $\sum_{j\in T_k''}y_j=(6\zeta^2+2\zeta+1)!$ for all $k\in [m_3]$. Given that $y_j=\sum_{i\in T_j'}x_i$, we let $T_k=\{i:\textnormal{there exists some } j\in T_k'' \textnormal{ such that } i\in T_j'\}$, then it is clear that $\sum_{i\in T_k}x_i=(6\zeta^2+2\zeta+1)!$. Further, $y_j=0$ for $j>m_2$. We simply let $T_j'$ be $T_{j-m_2+m_3}$. Now it is easy to verify that we obtain a partition of $[m]$ into $m'=m_1-m_2+m_3$ subsets $T_k$'s such that $|T_k|\le (6\zeta^2+2\zeta+1)!\cdot (6\zeta+2)=2^{\OO(\zeta^2\log\zeta)}$, and $\sum_{i\in T_k}x_i\in \{0,(6\zeta^2+2\zeta+1)!\}$. Hence, Lemma~\ref{lemma:number} is proved. 
\end{proof}

\subsection{Proof of Lemma~\ref{lemma:dec-1}}
\begin{T7}
Let $\veg\in \ker_{\Z}(H_{\textnormal{com}})$. Let 
\begin{eqnarray*}
\lambda=(6\lambda_0^2+2\lambda_0+1)!=2^{2^{2^{\OO(t_B^2\log \Delta)}}}, \textnormal{ where } \lambda_0:=\Delta t_B g_{\infty}(H_{\textnormal{com}}^{\textnormal{two-stage}})=2^{2^{\OO(t_B^2\log \Delta)}}.
\end{eqnarray*}
If $B\veg^0$ is a multiple of $\lambda$, then $\veg$ admits a uniform decomposition $\veg=\sum_{j=1}^N\veeta_j$ such that $\|\veeta_j\|_{\infty}\le 2^{2^{2^{\OO(t_B^2\log \Delta)}}}$. Furthermore, $B\veeta_j^0$ is a multiple of $\lambda$ for all $j$.
\end{T7}
\begin{proof}
As $\veg\in \ker_{\Z}(H^{\textnormal{two-stage}}_{\textnormal{com}})$, there exists some integer $L'$ and ${\ve\xi}_j$'s such that:
$\ve g={\ve\xi}_1+{\ve\xi}_2+\cdots+{\ve\xi}_{L'},$
where for all $j$ it holds that $H^{\textnormal{two-stage}}_{\textnormal{com}}{\ve\xi}_j=\ve 0$, $\|{\ve\xi}_j\|_\infty\le g_{\infty}(H^{\textnormal{two-stage}}_{\textnormal{com}})$, and ${\ve\xi}_j\sqsubseteq \veg$. Consequently, $B\ve\xi_j^0\le \Delta t_Bg_{\infty}(H^{\textnormal{two-stage}}_{\textnormal{com}})=\lambda_0$. 
%We know 
%\begin{eqnarray}\label{eq:dec}
%\lambda\veg=\underbrace{\ve\xi_1+\ve\xi_1+\cdots+\ve\xi_1}_{\lambda}+\underbrace{\ve\xi_2+\ve\xi_2+\cdots+\ve\xi_2}_{\lambda}+\cdots+\underbrace{\ve\xi_{L'}+\ve\xi_{L'}+\cdots+\ve\xi_{L'}}_{\lambda}.    
%\end{eqnarray}
%For simplicity we rewrite the above equation as $\lambda\veg=\sum_{j=1}^{\lambda L'}\bar{\ve\xi}_j$ where each $\bar{\ve\xi}_j$ is some $\xi_{j'}$ in Eq~\eqref{eq:dec}.
Consider the sequence $B{\ve\xi}_1^0, B{\ve\xi}_2^0,\cdots,$ $B{\ve\xi}_{L'}^0$. %For simplicity let this sequence be $x_1,x_2,\cdots,x_{\lambda L'}$. 
It is clear that $|B{\ve\xi}_j^0|\le \lambda_0$ and $\sum_jB{\ve\xi}_j^0= B\veg^0$ is a multiple of $\lambda$. According to Lemma~\ref{lemma:number}, we can partition $[ L']$ into $m'$ subsets $T_1,T_2$, $\cdots$, $T_{m'}$ such that $\bigcup_{k=1}^{m'}T_k=[m]$, and for all $k\in [m']$ it holds that $|T_k|\le 2^{\OO(\lambda_0^2\log\lambda_0)}$, $\sum_{j\in T_k}B{\ve\xi}^0_j\in \{0, \lambda \cdot sgn( B\veg^0)\}$. Let $\veeta_j=\sum_{j\in T_k}{\ve\xi}_j$. According to the definition in Eq~\eqref{eq5-extra}, we get that $\veg=\sum_{j=1}^{m'}\veeta_j$ is a uniform decomposition. 
\end{proof}

\subsection{Extension of Theorem~\ref{11-n} to almost combinatorial 4-block $n$-fold IP}\label{appsec:almost}
The extension of Theorem~\ref{11-n} to almost combinatorial $4$-block $n$-fold IP is straightforward. For the completeness of the paper, we give the formal proof below. 

When considering such an $s_B\times t_B$ matrix $B$ with rank 1, we can always transform $B$ into $\bar{B}$, in which  the first row is $\ver_1^{\top}$, and all the other rows are $\ve0$. It implies that when $\text{rank}(B)=1$, it is sufficient to consider such a case $B=(\ver_1,\ve 0,\ldots,\ve 0)^{\top}$, where $\ver_1\neq \ve 0$. 

Thus, without loss of generality, we assume that all almost combinatorial $4$-block $n$-fold matrices always have the common feature that $B=(\ver_1,\ve 0,\ldots,\ve 0)^{\top}$, where $\ver_1\neq \ve 0$. From now on we denote by $\tilde{H}_{\textnormal{com}}$ an almost combinatorial $4$-block $n$-fold matrix, and by $\tilde{H}^{\textnormal{two-stage}}$ the two-stage stochastic matrix obtained by removing $(C,D_1,\cdots,D_n)$ from $\tilde{H}_{\textnormal{com}}$.

We first have a similar result to Lemma~\ref{lemma:dec-1}.
\begin{lemma}\label{lemma:dec-2}
Let $\veg\in \ker_{\Z}(\tilde{H}_{\textnormal{com}})$. Let 
\begin{eqnarray*}
\lambda=(6\lambda_0^2+2\lambda_0+1)!=2^{2^{2^{\OO(t_B^2\log \Delta)}}}, \textnormal{ where } \lambda_0:=\Delta t_B g_{\infty}(H_{\textnormal{com}}^{\textnormal{two-stage}})=2^{2^{\OO(t_B^2\log \Delta)}}.
\end{eqnarray*}
Let $B\veg^0=(B\veg^0[1],0,\cdots,0)$. If $B\veg^0[1]$ is a multiple of $\lambda$, then $\veg$ admits a uniform decomposition $\veg=\sum_{j=1}^N\veeta_j$ such that $\|\veeta_j\|_{\infty}\le 2^{2^{2^{\OO(t_B^2\log \Delta)}}}$. Furthermore, $B\veeta_j^0=(B\veg^0_j[1],0,\cdots,0)$ where $B\veg^0_j[1]$ is a multiple of $\lambda$ for all $j$.
\end{lemma}
\begin{proof}
As $\veg\in \ker_{\Z}(\tilde{H}^{\textnormal{two-stage}}_{\textnormal{com}})$, there exists some integer $L'$ and ${\ve\xi}_j$'s such that:
$\ve g={\ve\xi}_1+{\ve\xi}_2+\cdots+{\ve\xi}_{L'},$
where for all $j$ it holds that $\tilde{H}^{\textnormal{two-stage}}_{\textnormal{com}}{\ve\xi}_j=\ve 0$, $\|{\ve\xi}_j\|_\infty\le g_{\infty}(\tilde{H}^{\textnormal{two-stage}}_{\textnormal{com}})$, and ${\ve\xi}_j\sqsubseteq \veg$. Consequently, $B\ve\xi_j^0=(B\ve\xi_j^0[1],0,\cdots,0)$ where $B\ve\xi_j^0[1]\le \Delta t_Bg_{\infty}(\tilde{H}^{\textnormal{two-stage}}_{\textnormal{com}})=\lambda_0$. 
%We know 
%\begin{eqnarray}\label{eq:dec}
%\lambda\veg=\underbrace{\ve\xi_1+\ve\xi_1+\cdots+\ve\xi_1}_{\lambda}+\underbrace{\ve\xi_2+\ve\xi_2+\cdots+\ve\xi_2}_{\lambda}+\cdots+\underbrace{\ve\xi_{L'}+\ve\xi_{L'}+\cdots+\ve\xi_{L'}}_{\lambda}.    
%\end{eqnarray}
%For simplicity we rewrite the above equation as $\lambda\veg=\sum_{j=1}^{\lambda L'}\bar{\ve\xi}_j$ where each $\bar{\ve\xi}_j$ is some $\xi_{j'}$ in Eq~\eqref{eq:dec}.
Consider the sequence $B{\ve\xi}_1^0[1], B{\ve\xi}_2^0[1],\cdots,$ $B{\ve\xi}_{L'}^0[1]$. %For simplicity let this sequence be $x_1,x_2,\cdots,x_{\lambda L'}$. 
It is clear that $|B{\ve\xi}_j^0[1]|\le \lambda_0$ and $\sum_jB{\ve\xi}_j^0[1]= B\veg^0[1]$ is a multiple of $\lambda$. According to Lemma~\ref{lemma:number}, we can partition $[ L']$ into $m'$ subsets $T_1,T_2$, $\cdots$, $T_{m'}$ such that $\bigcup_{k=1}^{m'}T_k=[m]$, and for all $k\in [m']$ it holds that $|T_k|\le 2^{\OO(\lambda_0^2\log\lambda_0)}$, $\sum_{j\in T_k}B{\ve\xi}^0_j[1]\in \{0, \lambda \cdot sgn( B\veg^0)\}$. Let $\veeta_j=\sum_{j\in T_k}{\ve\xi}_j$. According to the definition in Eq~\eqref{eq5-extra}, we get that $\veg=\sum_{j=1}^{m'}\veeta_j$ is a uniform decomposition. 
\end{proof}

\begin{theorem}\label{almost-11-n}
 Let $\tilde{H}_{\textnormal{com}}$ be an almost combinatorial $4$-block $n$-fold matrix. Then there exists a positive integer $\lambda\le 2^{2^{2^{\OO(t_B^2\log (t_B\Delta))}}}$ (which is only dependent on $t_B$ and $\Delta$) such that for any $\ve g\in\ker_{\Z}(H_{\textnormal{com}})$, we have $\lambda\veg=\veg_1+\veg_2+\cdots+\veg_p$ for some $p\in \mathbb{Z}_{>0}$ and $\veg_j\in \ker_{\Z}(\tilde{H}_{\textnormal{com}})$, and furthermore, $\veg_j\sqsubseteq \lambda\veg$ and $\|\veg_j\|_\infty=2^{2^{\OO(s_At_A\log\Delta+s_Dt_D\log\Delta)}\cdot 2^{2^{\OO(t_B^2\log\Delta)}}}$.
\end{theorem}
\begin{proof}
Consider any $\veg\in \ker_{\Z}(\tilde{H}_{\textnormal{com}})$. Clearly $B(\lambda\veg^0)=\lambda B\veg^0=(\lambda B\veg^0[1],0,\cdots,0)$ where $\lambda B\veg^0[1]$ is a multiple of $\lambda$, thus by Lemma~\ref{lemma:dec-2}, $\lambda\veg$ admits a uniform decomposition $\lambda\veg = \sum_{j=1}^N\veeta_j$ where {$\|\veeta_j\|_{\infty}\le \eta_{\max}=2^{2^{2^{\OO(t_B^2\log \Delta)}}}$ and every $B\veeta_j^0=(B\veeta_j^0[1],0,\cdots,0)$ where $B\veeta_j^0[1]$ is a multiple of $\lambda$. } 

If this decomposition is not $\omega$-balanced for $\omega\le (\Delta t_D\eta_{\max})^{\OO(s_D^2)}$, then by Lemma~\ref{lemma:balance} we obtain $\veeta\sqsubseteq \lambda\veg$ with $\|\veeta\|_{\infty}\le (\Delta t_D\eta_{\max})^{\OO(s_D^2)}$, $\veeta\in \ker_{\Z}(\tilde{H}_{\textnormal{com}})$ and $B(\lambda\veg^0-\veeta^0)=\ve 0$. $B\veeta^0[1]$ is a multiple of $\lambda$.
Otherwise this decomposition is $\omega$-balanced. By Lemma~\ref{lemma:all-balance}, we can obtain a uniform decomposition $\lambda\veg = \sum_{j=1}^{N'}{\veeta}_j'$ such that $\max_j\|{\veeta}_j'\|\le \omega\eta_{\max}$ and all $\veeta_j'$'s are tier-1.  According to Lemma~\ref{lemma:sub},
if $\lambda\|\veg\|_{\infty}>\tau$ for $\tau=(\omega\Delta\eta_{\max})^{\Delta^{\OO(s_At_A+s_Dt_D)}}=2^{2^{\OO(s_At_A\log\Delta+s_Dt_D\log\Delta)}\cdot 2^{2^{\OO(t_B^2\log\Delta)}}}$, then we are able to find some $\veeta\sqsubseteq \lambda\veg$ such that $\tilde{H}_{\textnormal{com}}\veeta=\ve 0$, $\|\veeta\|_\infty=\OFPT(1)$ and $\veeta^0=\sum_{j\in S}\veeta_j^0$ for some $S\subseteq [N]$. As every $B\veeta_j^0=(B\veeta_j^0[1],0\cdots,0)$ satisfies that $B\veeta_j^0[1]$ is a multiple of $\lambda$, we know $B\veeta^0=(B\veeta^0[1],0,\cdots,0)$  where $B\veeta^0[1]$ is also a multiple of $\lambda$. In both cases, we find $\veeta\sqsubseteq \lambda\veg$ where $B\veeta^0=(B\veeta^0[1],0,\cdots,0)$, and $B\veeta^0[1]$ is a multiple of $\lambda$.

Now consider $\lambda\veg-\veeta$. Obviously $\lambda\veg-\veeta\in \ker_{\Z}(\tilde{H}_{\textnormal{com}})$. It is easy to see $B(\lambda\veg^0-\veeta^0)=(x,0,\cdots,0)$ where $x$ is a multiple of $\lambda$. Thus, if $\|\lambda\veg-\veeta\|_{\infty}> \tau$ we can continue to decompose $\lambda\veg-\veeta$ using our argument above. Hence, Theorem~\ref{almost-11-n} is proved. 
%{\color{red}If $\|\lambda\veg-\veeta\|_{\infty}$ is } too large, then we can continue to find out $\veeta'\sqsubseteq \lambda\veg-\veeta$ such that $\|\veeta'\|_\infty=\OFPT(1)$. Iteratively applying the above argument, Theorem~\ref{11-n} is proved.  
\end{proof}

\section{Omitted contents in Section~\ref{sec:alg}}\label{appsec:alg}
%Before we present our algorithms, we provide some preliminary on algorithms for integer programming.
The goal of this section is to develop algorithms for combinatorial 4-block $n$-fold IP. Towards this, we first bound the infinity norm of Graver basis elements.

\begin{T8}
Let $\veg\in\G(H_{\textnormal{com}})$ be a Graver basis element, then $\|\ve g\|_\infty= g_{\infty}(H_{\textnormal{com}})$ where $g_{\infty}(H_{\textnormal{com}})\le 2^{2^{\OO(t_A\log\Delta+s_Dt_D\log\Delta)}\cdot 2^{2^{\OO(t_B^2\log\Delta)}}}\cdot n=\OFPT(n).$
\end{T8}
\begin{proof}
According to Theorem~\ref{11-n}, we know for $\lambda=\OFPT(1)$ there exist $\veg_j\in\ker_{\Z}(H_{\textnormal{com}})$ such that $\lambda\veg=\sum_{j=1}^p \veg_j$, $\veg_j\sqsubseteq\lambda \veg$ and  $\|\veg_j\|_\infty=\OFPT(1)$.
To show $\|\veg\|_{\infty}=\OFPT(n)$, it suffices to show that $p=\OFPT(n)$. Note that if any $\veg_j\sqsubset \veg$, then it will violate the fact that $\veg$ is a Graver basis element. Let $\vex[h]$ denote the $h$-th coordinate of a vector $\vex$. We know $\veg_j\not\sqsubset \veg$ implies that there exists some $h_j$-th coordinate such that $|\veg_j[h_j]|\ge |\veg[h_j]|$, and we call $h_j$ as the critical coordinate of $\veg_j$. If there are multiple critical coordinates, we pick an arbitrary one. Now we have a list of critical coordinates $h_1,h_2,\cdots,h_p$ where $1\le h_j\le t_B+nt_A$. We claim that every index $k\in [t_B+nt_A]$ can occur at most $\lambda$ times in the list. Supposing on the contrary some index $k$ appears $\lambda+1$ or more times, then there exist $\veg_{j_1},\veg_{j_2},\cdots,\veg_{j_{\lambda+1}}$ where everyone's $k$-th coordinate has an absolute value no less than $|\veg[k]|$. However, $\sum_{\ell=1}^{\lambda+1}\veg_{j_{\ell}}\sqsubseteq \lambda\veg$ implies that the summation of the absolute value of their $k$-th coordinates is bounded by $|\lambda\veg[k]|$, which is a contradiction. Hence, every index occurs at most $\lambda$ times in the list, implying that $p\le \lambda(t_B+nt_A)$. Hence, Theorem~\ref{coro:graver} is proved. More precisely, $$g_{\infty}(H_{\textnormal{com}})\le 2^{2^{\OO(t_A\log\Delta+s_Dt_D\log\Delta)}\cdot 2^{2^{\OO(t_B^2\log\Delta)}}}\cdot n.$$
\end{proof}
%Let $g_{\infty}(H_{\textnormal{com}})$ denote the upper bound on the Graver basis elements for combinatorial $4$-block $n$-fold IP, Theorem~\ref{coro:graver} shows that 
%$$g_{\infty}(H_{\textnormal{com}})\le 2^{2^{\OO(t_A\log\Delta+s_Dt_D\log\Delta)}\cdot 2^{2^{\OO(t_B^2\log\Delta)}}}\cdot n.$$
\begin{remark*}
As Theorem~\ref{11-n} remains true for almost combinatorial $4$-block $n$-fold IP where $\text{rank}(B)=1$, Theorem~\ref{coro:graver} is also true for almost combinatorial $4$-block $n$-fold IP. Denote by $\tilde{H}_{\textnormal{com}}$ the constraint matrix of almost combinatorial $4$-block $n$-fold IP, and denote by $g_{\infty}(\tilde{H}_{\textnormal{com}})$ the upper bound on its Graver basis elements, then we have
$$g_{\infty}(\tilde{H}_{\textnormal{com}})\le 2^{2^{\OO(s_At_A\log\Delta+s_Dt_D\log\Delta)}\cdot 2^{2^{\OO(t_B^2\log\Delta)}}}\cdot n.$$
\end{remark*}

Now we are ready to design FPT algorithms for combinatorial 4-block $n$-fold IP using the iterative augmentation framework.

\subsection{Linear Objective Functions}\label{linear_obj}
\begin{theorem}\label{thm:linear}
Combinatorial 4-block $n$-fold IP with a linear objective function $f(\vex)=\vew\vex$ can be solved in time: $$ 2^{2^{\OO(t_A\log\Delta+s_Dt_D\log\Delta)}\cdot 2^{2^{\OO(t_B^2\log\Delta)}}}\cdot n^{5+o(1)}=\OFPT(n^{5+o(1)}).$$
\end{theorem}

\begin{proof}
Utilizing the idea of approximate Graver-best oracle introduced by Altmanov{\'a} et al.~\cite{altmanova2019evaluating} and implicitly by Eisenbrand et al.~\cite{eisenbrand2018fastera}, it is sufficient that for every $\rho=2^0,2^1,2^2,\cdots,2^{h}$ where $h=\OO(n^{1+o(1)}\log \Delta)$,\footnote{Here $h\le \log \|\veu-\vel\|_{\infty}$. However, utilizing the techniques of Tardos~\cite{tardos1986strongly}, Koutecký et al.~\cite{koutecky2018parameterized} showed that without loss of generality $\|\veb\|_\infty, \|\vel\|_\infty, \|\veu\|_\infty\le 2^{\OO(n\log n)}\Delta^{\OO(n)}$.} we find out an augmentation of the form $\rho\vey$ which is no worse than $\rho\veg$ for any Graver basis element $\veg$ (i.e., $\rho\vey$ gives an improvement to the objective value larger than or equal to any $\rho\veg$). Observing that $B\veg^0\in \Z$ and $|B\veg^0|\le t_B\Delta g_{\infty}(H_{\textnormal{com}})$ for every Graver basis element, we  consider the following IP$(\rho,\phi)$ for every fixed $\rho$ and $\phi\in [-t_B\Delta g_{\infty}(H_{\textnormal{com}}): t_B\Delta g_{\infty}(H_{\textnormal{com}})]$:
 \begin{eqnarray}\label{eq34}
 	\min\{\vew\cdot \vey: H_{\textnormal{com}} \vey=\ve0, \vel\le \vex_0+\rho \vey \le \veu, B\vey^0=\phi, \vey\in \Z^{t_B+nt_A} \}.
 \end{eqnarray}
 It is clear that the optimal solution $\vey^*(\rho,\phi)$ to IP~\eqref{eq34} is no worse than $\rho\veg$ for any Graver basis element satisfying that $B\veg^0=\phi$. Taking the best solution out of all $\vey^*(\rho,\phi)$ gives the desired augmentation. 
 
 We write down explicitly the constraints of IP$(\rho,\phi)$ as follows: 
\begin{eqnarray*}
	&& C\vey^0+  \sum_{i=1}^{n}D_i\vey^i=\ve 0\\
	&& B\vey^0=\phi\\
	&& A_i\vey^i=-\phi,\quad \forall 1\le i\le n 
\end{eqnarray*}
The constraint matrix $H_{(1)}$ is as follows:  
\begin{eqnarray*}
	H_{(1)}=%\FourBlockBig[n]CDBA :=
	\begin{pmatrix}
		C & D_1 & D_2 & \cdots & D_n \\
		B & 0 & 0  &  \cdots &  0 \\
		0 & A_1  &  &   &   \\
		0 &   &  A_2 &  &   \\
		\vdots &   &   & \ddots &   \\
		0 &   &   &   & A_n
	\end{pmatrix}.\label{eq32}
\end{eqnarray*}

Hence, IP$(\rho,\phi)$ is a generalized $n$-fold IP. Using the algorithm of Cslovjecsek et al.~\cite{cslovjecsek2021block}, it can be solved in time $2^{\OO(s^2_As_D)}(s_Ds_A\Delta)^{\OO(s_A^2+s_As_D^2)} (nt_A)^{1+o(1)}$. 

The number of augmenting steps can be bounded by $\OO(n^{2+o(1)})$~\cite{de2012algebraic}, and in each augmentation we have to solve IP$(\rho,\phi)$ for all $\rho$ and $\phi$.
Thus the overall running time is 
$$ 2^{2^{\OO(t_A\log\Delta+s_Dt_D\log\Delta)}\cdot 2^{2^{\OO(t_B^2\log\Delta)}}}\cdot n^{5+o(1)}=\OFPT(n^{5+o(1)}).$$
%Then pick the best solution from all $\OFPT(n)$ IPs depending on the value of $\phi$. Notice that $\rho$ may take {\color{red}$\OFPT(n\log n)$} distinct values, and hence the total time is $\OFPT(n^3\log n)$.
\end{proof}

Using the same argument but replacing $g_{\infty}(H_{\textnormal{com}})$ with $g_{\infty}(\tilde{H}_{\textnormal{com}})$, we have the following:
\begin{corollary}\label{coro:linear}
Almost combinatorial 4-block $n$-fold IP with a linear objective function $f(\vex)=\vew\vex$ can be solved in time: $$ 2^{2^{\OO(s_At_A\log\Delta+s_Dt_D\log\Delta)}\cdot 2^{2^{\OO(t_B^2\log\Delta)}}}\cdot n^{5+o(1)}=\OFPT(n^{5+o(1)}).$$
\end{corollary}

%By Lemma~\ref{lemma21}, if we know a Graver-best oracle, {\color{red}IP} can be solved in strongly polynomial oracle time. 
%Now since $\|\ve g\|_\infty\le \OFPT(n)$, we show that the focused problem can be solved  in $\OFPT(n^2)$ time. 

%If we go through all elements in $\G(H)$, the Graver-best step $\veh$ satisfying the following Eq~\eqref{eq33} can be found. 

%That is, the Graver-best step $\veh$ is better than any one Graver basis. 
%Furthermore, the Graver-best step $\veh$ must be better than any Graver basis $\veg$ with $\|\ve g\|_\infty\le \OFPT(n)$. 
%{\color{red}Equivalently}, we consider the following equality:

%\begin{remark*}
%In the time $\OFPT(n^3\log n)$, $\OFPT$ hides the factor $2^{O(s_D)}(s_D\Delta)^{O(s_D^2)}t^2_At_B\cdot  ((s_D+5^{t_A}\Delta^{t_A+s_Dt_A})(t_B\Delta)^{O(t_B\Delta^{t_B^2})})^{O(s_Dt_A+
%\Delta^{t_A+s_Dt_A} t_A5^{t_A})}\Delta^{t_A+s_Dt_A}$.
%\end{remark*}

\subsection{Separable Convex Objective Functions}\label{separable}
We consider a separable convex objective function. A convex function $f: \mathbb{R}^{t_B+nt_A}\rightarrow \mathbb{R}$ is called separable if there are convex functions $f^i_j: \mathbb{R}\rightarrow \mathbb{R}$ such that $f(\ve x)=\sum_{i=0}^n f^i(\ve x^i)=\sum_{j=1}^{t_B} f^0_j(x^0_j)+\sum_{i=1}^n\sum_{j=1}^{t_A} f^i_j( x^i_j)$. Henceforth, we consider the problem 
\begin{eqnarray}\label{ILP_np2}
\min\{ f(\ve x): H_{\textnormal{com}} \vex=\veb, \vel\le \vex\le \veu, \vex\in \Z^{t_B+nt_A} \},
\end{eqnarray}
\iffalse

where 
\begin{eqnarray}
{H_{\textnormal{com}}}:=
\begin{pmatrix}
C & D_1 & D_2 & \cdots & D_n \\
B & A_1 & 0  &   & 0  \\
B & 0  & A_2 &   & 0  \\
\vdots &   &   & \ddots &   \\
B & 0  & 0  &   & A_n
\end{pmatrix} \enspace,
\end{eqnarray}
and $B\in\Z^{1\times t_B} $.

\fi

We assume that the objective function $f$ is presented by an \emph{evaluation oracle} that, when queried on a vector $\ve x$, returns the values $f^i(\ve x^i)$ for all $i = 0,1,\ldots,n$. The time complexity now measures the number of arithmetic operations and oracle queries.

%When the objective function $f$ is separable convex mapping $\mathbb{Z}^{t_B+nt_A}$ to $\mathbb{Z}$, we show a result similar to Theorem~\ref{thm9}.

\begin{T2}
 Consider combinatorial 4-block $n$-fold IP with a separable convex objective function $f$ mapping $\Z^{t_B+nt_A}$ to $\Z$. Let $P$ be the set of feasible integral points, and let $\hat{f}:=\max_{x,y\in P}(f(x)-f(y))$. Then~\eqref{ILP_np2} can be solved in $\OFPT(n^4\hat{L}^2 \log^2(\hat{f}))$ time. 
 More specifically, the running time is 
$$2^{2^{\OO(t_A\log\Delta+s_Dt_D\log\Delta)}\cdot 2^{2^{\OO(t_B^2\log\Delta)}}}\cdot n^{4}\hat{L}^2\log^2(\hat{f}),$$ where $\hat{L}$ denotes the logarithm of the largest number occurring in the input.
\end{T2}
\begin{proof}
We use a similar idea as Theorem~\ref{thm:linear}. It has been shown by Eisenbrand et al.~\cite{eisenbrand2018fastera} that for a separable convex function, it is still sufficient that for every $\rho=2^0,2^1,2^2,\cdots,2^{\OO(\log \|\veu-\vel\|_{\infty}))}$, we find out an augmentation of the form $\rho\vey$ which is no worse than $\rho\veg$ for any Graver basis element $\veg$. Hence, similarly, for every $\rho$ and every $\phi\in [-t_B\Delta g_{\infty}(H_{\textnormal{com}}): t_B\Delta g_{\infty}(H_{\textnormal{com}})]$, we solve the following:
\begin{eqnarray*}
 	\min\{f(\vex_0+\rho \vey)-f(\vex_0): H_{\textnormal{com}} \vey=\ve0, \vel\le \vex_0+\rho \vey \le \veu, B\vey^0=\phi, \vey\in \Z^{t_B+nt_A}\}.
 \end{eqnarray*}
The above is a generalized $n$-fold IP with a separable convex objective function, which can be solved in time of $n^2t^2\hat{L} \log(\hat{f}) (s_D\Delta)^{\OO(s_D^2)}$~\cite{eisenbrand2018fastera}, where $t=\max\{t_A,t_B\}$. 

The number of augmenting steps can be bounded by $(2n-2)\log(\hat{f})$~\cite{de2012algebraic}. Hence, the overall running time is 
$$ 2^{2^{\OO(t_A\log\Delta+s_Dt_D\log\Delta)}\cdot 2^{2^{\OO(t_B^2\log\Delta)}}}\cdot n^{4}\hat{L}^2\log^2(\hat{f})=\OFPT(n^{4}\hat{L}^2\log^2(\hat{f})).$$
\end{proof}
 
 Using the same argument but replacing $g_{\infty}(H_{\textnormal{com}})$ with $g_{\infty}(\tilde{H}_{\textnormal{com}})$, we have the following:

\begin{corollary}
\label{coro:convex}
Consider almost combinatorial 4-block $n$-fold IP with a separable convex objective function $f$ mapping $\Z^{t_B+nt_A}$ to $\Z$. Let $P$ be the set of feasible integral points for~\eqref{ILP_np2}, and let $\hat{f}:=\max_{x,y\in P}(f(x)-f(y))$. Then~\eqref{ILP_np2} can be solved in $\OFPT(n^{4}\hat{L}^2 \log^2(\hat{f}))$ time. 
 More specifically, the running time is 
$$2^{2^{\OO(s_At_A\log\Delta+s_Dt_D\log\Delta)}\cdot 2^{2^{\OO(t_B^2\log\Delta)}}}\cdot n^{4}\hat{L}^2\log^2(\hat{f}),$$ where $\hat{L}$ denotes the logarithm of the largest number occurring in the input.
\end{corollary}

\section{Omitted contents in Section~\ref{appli}}\label{appsec:application}
Scheduling is a fundamental problem in operations research and computer science. The classical scheduling problem as well as its generalizations have been studied extensively in the literature. In particular, {approximation algorithms have been developed  for scheduling with rejection cost (see, e.g.,~\cite{engels2003techniques,hoogeveen2003preemptive,sviridenko2013approximating})}, and scheduling with the bicriteria of makespan and (weighted) total completion time (see, e.g.,~\cite{allahverdi2008two,allahverdi2002no,cheng2015two,xiong2014meta}.
In recent years, FPT algorithms have been developed for the classical scheduling problems~\cite{knop2018scheduling,mnich2015scheduling}. However, not much is known regarding how these algorithms can be generalized to deal with more sophisticated scheduling models. In particular, FPT algorithms have been developed for single machine scheduling with rejection cost~\cite{mnich2015scheduling}, while FPT algorithms for parallel machines are still unknown. FPT algorithms for bicriteria scheduling are also unknown. 

In this section, we show that combinatorial 4-block $n$-fold IP offers a strong tool for dealing with these generalizations on the classical scheduling problems.

%There are many literatures concerning on the bicriteria of makespan and total completion time (see, e.g.~\cite{allahverdi2002no,cheng2015two,allahverdi2008two,xiong2014meta}. and the bicriteria of total weighted completion time and rejection cost~\cite{mnich2015scheduling}.  
\subsection{Scheduling with rejection}\label{sche1}
We restate our problem $R||C_{\max}+E$ here. Given are $m$ machines and $k$ different types of jobs, with $N_j$ jobs of type $j$. A job of type $j$ has a processing time of $p^i_j\in\Z_{\ge 0}$ if it is processed by machine~$i$. Every job of type $j$ also has a rejection cost $u_j$. A job is either processed on one of the machine, or is rejected. The goal is to minimize the makespan $C_{\max}$ plus the total rejection cost $E$, where makespan denotes the largest job completion.

FPT algorithms for scheduling with rejection has been considered by Mnich and Wiese~\cite{mnich2015scheduling}. However, they considered single machine scheduling with rejection. We are not aware of FPT algorithms for parallel machine scheduling with job rejection cost.

%in 2015 considered the below model of scheduling with rejection: $n$ non-preemptive jobs released at time zero with $k$ types; each job $j$ is characterized by a processing time $p_j$, a weight $w_j$ and rejection cost $u_j$; the goal is to reject a set of at most $n_0$ jobs and schedule the other jobs on a single machine, so as to minimize the sum of weighted completion times plus the total rejection cost. They proved that the problem can be solved in FPT time of $O(n(n_0+1)^k+n\log n)$ when $n_0$ and $k$ are as parameters;  besides, let $k_0$ be the number of different types of $\frac{p_j}{w_j}$, so when making $k_0$ as a parameter, this problem could be solved in FPT time of $(2k_0n+\log (\max_j\{u_j,p_j,w_j\}))^{O(1)}\cdot 2^{O(k_0^3)}$. In Section~\ref{appli} we will consider one different model of scheduling with rejection on multiple machines and get an FPT algorithm of two parameters. 

The goal of this subsection is to prove the following.
\begin{T9}
$R||C_{\max}+E$ can be solved in
 $m^{5+o(1)} 2^{2^{\OO(k^2\log p_{\max})}\cdot 2^{2^{\OO(\log p_{\max})}}}+|I|$ time, where $|I|$ denotes the length of the input.
\end{T9}
\begin{proof}
We model the scheduling problem with rejection cost as a combinatorial 4-block $n$-fold IP to solve it. 
Let $x^i_j\in\Z_{\ge0}$ denote the total number of jobs of type $j$ assigned to machine~$i$ in a schedule, and $C_{\max}$ be the makespan. Then we have the following IP$_{\textnormal{sche1}}$: 
	\begin{eqnarray}
	&\min{ }& C_{\max}+\sum_{j=1}^{k}u_j(N_j-\sum_{i=1}^{m}x^i_j)\nonumber\\
	&&\sum_{i=1}^{m}x^i_j\le N_j,\hspace{3.7cm} \forall 1\le j\le k\label{eq40}\\
	&& \sum_{j=1}^{k}p^i_jx^i_j-C_{\max}\le 0,\quad \hspace{2.1cm}\forall 1\le i\le m\label{eq41}\\
	&& x^i_j\in\Z_{\ge 0} \nonumber
\end{eqnarray}
%where $x^i_j$ represents the number of jobs of type $j$ scheduled on machine~$i$. 
Here Constraint~\eqref{eq40} indicates that the total number of type-$j$ jobs being processed is at most $N_j$. Constraint~\eqref{eq41} indicates that the total job processing time on every machine is bounded by the makespan $C_{\max}$.
Let all variables be ordered as a vector $\vex=(\vex^1_{ \cdot},\vex^2_{\cdot},\ldots,\vex^m_{\cdot})$, where $\vex^i_{\cdot}=(x^i_{1},x^i_{2},\ldots,x^i_{k})$. It is easy to see that IP$_{\textnormal{sche1}}$ has the following constraint matrix:
	$$	H_{(3)}=%\FourBlockBig[n]CDBA :=
\begin{pmatrix}
		0 &	I & I & I & \cdots & I \\
		-1 &	\vep^1_{\cdot} & 0 & 0  & \cdots  & 0  \\
		-1 &	0 & \vep^2_{\cdot}  &0 &  \cdots & 0  \\
		-1 &	0 & 0  &\vep^3_{\cdot} &  \cdots & 0  \\
		\vdots &	\vdots &  \vdots & \vdots  & \ddots & \vdots  \\
		-1 &	0 & 0  & 0  & \cdots   &\vep^m_{\cdot}
\end{pmatrix},
$$
where $\vep^i_{\cdot}=(p^i_{1},p^i_{2},\ldots,p^i_{k})$. Hence, IP$_{\textnormal{sche1}}$ is a combinatorial $4$-block $n$-fold IP with a linear objective function. IP$_{\textnormal{sche1}}$ can be solved in $m^{5+o(1)} 2^{2^{\OO(k^2\log p_{\max})}\cdot 2^{2^{\OO(\log p_{\max})}}}+|I|$ time by using Theorem~\ref{thm:linear}, where $|I|$ denotes the input size of the given problem. More precisely, $|I|$ is bounded by $\OO(kp_{\max} (\max\{\log N_{\max}, \log u_{\max}\}))$ where $N_{\max}=\max_j N_j$ and $u_{\max}=\max_j u_j$. %When $k$ and $p_{\max}$ are parameters, IP$_{\textnormal{sche1}}$ could be solved in FPT time. 
\end{proof}
 
%This $H$ shows the form in~\eqref{eqblock}. 

\iffalse

%For each  constraint~\eqref{eq40}, we add $m$ slack variables; and for each constraint~\eqref{eq41}, we add one slack variable. Thus the above constraints after adding slack variables can still be written as the form in~\eqref{eqblock}, which is presented in the following: 

	$$	H'=%\FourBlockBig[n]CDBA :=
\begin{pmatrix}
	0 &	D & D & D & \cdots & D \\
	-1 &	\vep_{1\cdot} & 0 & 0  & \cdots  & 0  \\
	-1 &	0 & \vep_{2\cdot}  &0 &  \cdots & 0  \\
	-1 &	0 & 0  &\vep_{3\cdot} &  \cdots & 0  \\
	\vdots &	\vdots &  \vdots & \vdots  & \ddots & \vdots  \\
	-1 &	0 & 0  & 0  & \cdots   &\vep_{m\cdot}
\end{pmatrix},
$$where $\vep_{i\cdot}=(p_{i1},p_{i2},\ldots,p_{ik},1,0,\ldots, 0)$. The number of 0 components above is $k$. And
$$	{D}=
\begin{pmatrix}
	I &	\ve0 &  I \\
	 \ve0  &0	 & \ve0  \\
\end{pmatrix}.
$$

\fi

%We add $(k+1)m$ slack variables in total by Observation~\ref{obs}, and then get the new constraint matrix which is still kept the form of $H_{\textnormal{com}}$ in~\eqref{eq:matrix}. Finally, in the new constraint, $s_{D^\prime}=k$,  $t_{D^\prime}=t_{A^\prime}=2k+1$, $s_B=t_B=1$, $s_{A^\prime}=1$, $\Delta=p_{\max}=\max_{i,j}p^i_j$. IP$_{\textnormal{sche1}}$ can be solved in time of ${\color{blue}m^{5+o(1)} 2^{O(p_{\max}^{k^2}\cdot 2^{p_{\max}})}}\left({\color{red}\log (\max_j\{N_j,u_j\})}\right)$ according to our results. When $k$ and $p_{\max}$ are parameters, IP$_{\textnormal{sche1}}$ could be solved in FPT time. 

\begin{remark*}
One may suspect that IP$_{\textnormal{sche1}}$ can be solved through the generalized $n$-fold IP by guessing out the value of $C_{\max}$. However, this will require $p_{\max}\cdot \max_jN_j$ enumerations. 
%can not be solved in polynomial time by enumerating all the values of $C_{\max}$, whose number is up to $p_{\max}\cdot \max_jN_j$.
\end{remark*}

%It is worth mentioning that our algorithm also implies an FPT algorithm for the problem of scheduling with rejection parameterized by $p_{\max}$ and the number of machine types (instead of job types). 

\subsection{Scheduling with the objective of minimizing weighted completion time plus makespan}\label{sche2}
We restate our problem $R||\theta C_{\max}+\sum_{\ell}w_\ell C_\ell$ here. Given are $m$ machines and $k$ different types of jobs, with $N_j$ jobs of type $j$. A job of type $j$ has a processing time of $p^i_j\in\Z_{\ge 0}$ if it is processed by machine~$i$.
Each job $\ell$ of type $j$ also has a weight $w_j$, and the goal is to find an assignment of jobs to machines such that $\theta C_{\max}+\sum_{\ell}w_\ell C_\ell$ is minimized, where $C_\ell$ is the completion time of job $\ell$, $C_{\max}$ is the largest job processing time, and $\theta$ is a fixed input value.

FPT algorithms for $R||C_{\max}$ and $R||\sum_{\ell}w_\ell C_\ell$ have been developed by Knop and Kouteck{\`y}~\cite{knop2018scheduling}. However, their technique does not generalize to bicriteria as the natural IP formulation becomes 4-block $n$-fold, as we will show below.

The goal of this subsection is to prove the following.
\begin{T10}
 $R||\theta C_{\max}+\sum_{\ell}w_\ell C_\ell$ can be solved in
 $m^4 2^{2^{\OO(k^2\log p_{\max})}\cdot 2^{2^{\OO(\log p_{\max})}}}|I|^4$ time, where $|I|$ denotes the length of the input.
\end{T10}
\begin{proof}

Again we model the scheduling problem with combinatorial 4-block $n$-fold IP. Towards this,  we need to transform the objective function to a separable convex function. Such a transformation has been achieved by Knop and Kouteck{\`y}~\cite{knop2018scheduling}. For completeness of the paper, we briefly recap their transformation here.  %Let $f^i(\ve x^i)=\sum_{j=1}^{k}w_j C^i_j x^i_j$. 

Consider jobs scheduled on each machine~$i$.  Assume a set of jobs $J^i:=\{J_1,\ldots,J_h\}$ will be scheduled on the machine~$i$ such that $\delta_i(q)\ge \delta_i(q+1)$ for all $1\le q\le h-1$, where $\delta_i(q):=w_q/p^i_q$. We denote $\delta_i(h+1)=0$. It is clear that these jobs will be scheduled according to the Smith rule, and thus in the sequence of $J_1, J_2,\cdots, J_h$. Denote by $C^i_q$ the completion time of job $J_q$ on this machine~$i$. The following observation has been made in Lemma 2 of~\cite{knop2018scheduling}, $$\sum_{q=1}^{h}w_q C^i_q=\sum_{q=1}^{h}[\frac{1}{2}p^i(\{J_1,\ldots,J_q\})^2(\delta_i(q)- \delta_i(q+1))+\frac{1}{2}w_q p^i_q],$$ where $p^i(S)=\sum_{J_q\in S}p^i_q$. 

Now we are ready to set up an IP. We use $x^i_j$ to represent the number of jobs of type $j$ $(1\le j\le k)$ that are scheduled on machine~$i$ $(1\le i\le m)$, then the following holds:

\begin{lemma}[\cite{knop2018scheduling}, Corollary 1]
Given $x^i_1,\ldots,x^i_k$ representing the number of jobs of each type scheduled to run on machine~$i$, a permutation $\pi_i:[k]\rightarrow [k]$ such that $\delta_i(\pi_i(j) )\ge \delta_i(\pi_i(j+1))$ for all $1\le j\le k-1$ and $\delta_i(\pi_i(k+1))=0$, then $f^i(\ve x^i)=\frac{1}{2}\sum_{j=1}^{k}[(\sum_{h=1}^jp^i_hx^i_h)^2(\delta_i(\pi_i(j))- \delta_i(\pi_i(j+1)))+w_jp^i_jx^i_j]$. 
\end{lemma}

We introduce new variables as $z_j^i:=\sum_{h=1}^jp^i_hx^i_h$, then the objective function can be written as $\theta C_{\max}+\sum_{i=1}^{m}f^i(\ve x^i,\ve z^i)$, where $f^i(\ve x^i,\ve z^i)=\frac{1}{2}\sum_{j=1}^{k}[(z_j^i)^2(\delta_i(\pi_i(j))- \delta_i(\pi_i(j+1)))+w_jp^i_jx^i_j]$. 
Note that $f^i(\ve x^i,\ve z^i)$ is separable convex for any $i$ $(1\le i\le m)$~\cite{knop2018scheduling}. %Hence the objective function $\theta C_{\max}+\sum_{i=1}^{m}f^i(\ve x^i,\ve z^i)$ is separable convex. 

%Let $x^i_j\in\Z_{\ge0}$ denote the total number of jobs of type $j$ assigned to machine~$i$ in a schedule, $C_{\max}$ be the makespan. 
To summarize, we have the following IP$_{\textnormal{sche2}}$: 
\begin{eqnarray*}
&\min{ }& \theta C_{\max}+ \frac{1}{2}\sum_{i=1}^{m}\sum_{j=1}^{k}[(z_j^i)^2(\delta_i(\pi_i(j))- \delta_i(\pi_i(j+1)))+w_jp^i_jx^i_j]\\
	&&\sum_{i=1}^{m}x^i_j= N_j,\hspace{3.7cm} \forall 1\le j\le k\\
	&& \sum_{j=1}^{k}p^i_jx^i_j-C_{\max}\le 0,\quad \hspace{2.1cm}\forall 1\le i\le m\\
		&& \sum_{h=1}^jp^i_hx^i_h=z_j^i,\quad \hspace{3.1cm}\forall 1\le i\le m, 1\le j\le k\\
	&& x^i_j\in\Z_{\ge 0} 
\end{eqnarray*}
It is easy to verify that the constraint matrix is as follows:
\begin{eqnarray*}
H_{(4)}=%\FourBlockBig[n]CDBA :=
\begin{pmatrix}
0 & D_1 & D_2 & \cdots & D_m \\
B & A_1 & 0  &   & 0  \\
B & 0  & A_2 &   & 0  \\
\vdots &   &   & \ddots &   \\
B & 0  & 0  &   & A_m
\end{pmatrix},
\end{eqnarray*}
where
\begin{eqnarray*}
D_i=
\begin{pmatrix}
1 & 0 &  \cdots & 0  & 0 & 0 & \cdots & 0 \\
0 & 1 &   &    0     &  0 & 0 & \cdots & 0 \\
\vdots &   & \ddots &&  \vdots & \vdots & \ddots & \vdots  \\
0 &  0 &   & 1       &  0 & 0 & \cdots & 0 
\end{pmatrix},
&
A_i=
\begin{pmatrix}
p^i_1 & p^i_2 &  \cdots & p^i_k  & 0 & 0 & \cdots & 0 \\
p^i_1 & 0 &  \cdots &    0     & -1 & 0 & \cdots & 0 \\
p^i_1 &  p^i_2  &  &0 &  0 & -1 &  & 0  \\
\vdots &  \vdots &   &        &  \vdots &  & \ddots &  \\
p^i_1 & p^i_2 &  \cdots & p^i_k  & 0 & 0 &  & -1 
\end{pmatrix}
\end{eqnarray*}
and $B=(-1,\underbrace{0,\ldots,0}_{k})^{\top}$. 

This is an almost combinatorial 4-block $n$-fold IP. Using 
 Theorem~\ref{thm20}, the above IP can be solved in time $m^4 2^{2^{\OO(k^2\log p_{\max})}\cdot 2^{2^{\OO(\log p_{\max})}}}|I|^4$, where $|I|$ denotes the length of the input, which is bounded by $\OO(kp_{\max} (\max\{\log N_{\max},\log w_{\max}\}))$ where $N_{\max}=\max_j N_j$, $w_{\max}=\max_j w_j$.
 \end{proof}

% $\hat{f}:=\max_{x,y}(f(x)-f(y))$ and $\hat{L}$ denotes the logarithm of the largest number occurring in the input. It implies that when $k$ and $p_{\max}$ are parameters, IP$_{\textnormal{sche2}}$ can be solved in FPT time, which complements the previous results about scheduling with rejection  in~\cite{mnich2015scheduling}.

%\bibliography{B_block}

\end{document}